\documentclass[12pt,a4paper,oneside]{amsart}
\usepackage{amsmath,amssymb,amsthm,hyperref}
\usepackage{enumerate}
\usepackage{amsfonts}
\usepackage{amsthm}
\usepackage{amsmath}
\usepackage{amssymb}
\usepackage{cancel}
\usepackage{t1enc}
\usepackage{graphicx}
\usepackage{mathrsfs}
\usepackage{mathptmx}
\usepackage{lipsum}
\usepackage{tikz}
\usepackage{version} 
\usepackage{color}
\usepackage{enumerate}
\usepackage{datetime}
\numberwithin{equation}{section}
\usepackage[bf]{caption}     
\usepackage{geometry}        
\usepackage{bm}  
\usepackage{buczosierlyukas}
\usepackage{mathtools}
\usepackage[utf8]{inputenc}
\usepackage[normalem]{ulem}

\newtheorem{theorem}{Theorem}[section]
\newtheorem{proposition}[theorem]{Proposition}
\newtheorem{lemma}[theorem]{Lemma}

\DeclareMathAlphabet{\mathpzc}{OT1}{pzc}{m}{it}
{\theoremstyle{definition}
\newtheorem{definition}[theorem]{Definition}
\newtheorem{remark}[theorem]{Remark}

\newtheorem{question}[theorem]{Question}

}

\newcommand{\dub}{D_{\overline{B}*}}

\definecolor{aquam}{rgb}{0.5,1.0,1.0}
\definecolor{bbrown}{rgb}{0.75,0.38,0.15}
\definecolor{Cyan}{rgb}{0,0.6,0.6}
\definecolor{Darkblue}{rgb}{0,0,1}
\definecolor{Dodgerblue2}{rgb}{0,0.5,1}
\definecolor{Green}{rgb}{0,0.6,0.06}
\definecolor{Kahki}{rgb}{1,1,0.5}
\definecolor{Magenta}{rgb}{0.7,0,0.7}
\definecolor{bMagenta}{rgb}{1,.6,1}
\definecolor{Orange}{rgb}{0.8,0.3,0}
\definecolor{dOrchid}{rgb}{0.7,0.2,0.4}
\definecolor{Orchid}{rgb}{1,0.5,1}
\definecolor{Purple}{rgb}{0.65,0.07,0.85}
\definecolor{Royalblue}{rgb}{0.6,0.85,0.87}
\definecolor{Tan}{rgb}{0.54,0.42,0.23}
\definecolor{bTan}{rgb}{0.94,0.82,0.63}
\definecolor{zoltan}{rgb}{0,0.1,0.3}
\definecolor{Turquoise}{rgb}{0,0.85,0.87}
\definecolor{Yellow}{rgb}{1,1,0}
\definecolor{darkamber}{rgb}{0.4,0.19,0.28}
\definecolor{bYellow}{rgb}{1,1,0.6}
\definecolor{bRed}{rgb}{1,0.7,0.7}
\definecolor{boxcolb}{rgb}{0.87,0.77,0.75}
\definecolor{boxcol}{rgb}{0.6,0.85,0.87}
\definecolor{boxcolgreen}{rgb}{0.64,0.93,0.79}
\definecolor{boxcolaa}{rgb}{.75,.99,.70}
\definecolor{boxcolbb}{rgb}{0.39,0.50,0.56}
\definecolor{boxcolcc}{rgb}{1,0.81,0.65}
\definecolor{yy}{rgb}{0.43,0.21,.18}
\definecolor{gA}{gray}{0.5}
\definecolor{gB}{gray}{0.8}
\definecolor{gC}{gray}{0.9}

\newcommand{\xcx}{\color{black}}

\usepackage{xcolor}

\hypersetup{
  colorlinks   = true, 
  urlcolor     = blue, 
  linkcolor    = blue, 
  citecolor   = red 
}

\begin{document}

\title{Box dimension of generic H\"older level sets}
\author{Zolt\'an Buczolich$^*$}
\address[Z. Buczolich, B. Maga]{Department of Analysis, ELTE E\"otv\"os Lor\'and University, 
P\'azm\'any P\'eter S\'et\'any 1/c, 1117 Budapest, Hungary}
\address[B. Maga]{Alfr\'ed R\'enyi Institute of Mathematics, Re\'altanoda street 13-15, 1053 Budapest, Hungary}
\email{zoltan.buczolich@ttk.elte.hu}
\urladdr{http://buczo.web.elte.hu, ORCID Id: 0000-0001-5481-8797}

\author{Bal\'azs Maga$^\text{\textdagger}$}
\email{mbalazs0701@gmail.com}
\urladdr{  http://magab.web.elte.hu/}


\thanks{\scriptsize $^*$
 This author was also supported by the Hungarian National Research, Development and Innovation Office--NKFIH, Grant 124003.
}

\thanks{\scriptsize $^\text{\textdagger}$ This author was supported by the \'UNKP-22-3 New National Excellence of the Hungarian Ministry of Human Capacities, and by the 
Hungarian National Research, Development and Innovation Office-NKFIH, Grant 124749.
}


\date{\today}

\begin{abstract}
    Hausdorff dimension of level sets of generic continuous functions defined on fractals 
    can give information about the "thickness/narrow cross-sections" of a ‘‘network’’ corresponding to a fractal set, $F$.
    This lead to the definition of the topological Hausdorff dimension of fractals.
    Finer information might be obtained by considering the Hausdorff dimension of level sets of generic $1$-H\"older-$\aaa$ functions, 
    which has a stronger dependence on the geometry of the fractal, as displayed in our previous papers.
    In this paper, we extend our investigations to the lower and upper box-counting dimension as well: while the former yields results highly
    resembling the ones about Hausdorff dimension of level sets, the latter exhibits a different behaviour. 
	Instead of "finding narrow-cross sections", results related to upper box-counting dimension try to "measure" how much level sets can spread out on the fractal, how widely the generic function can "oscillate" on it.
    Key differences are illustrated by giving estimates concerning the Sierpiński triangle.
\end{abstract}

\maketitle

\setcounter{tocdepth}{3}

\tableofcontents


\section{Introduction}

Level sets of continuous functions defined on some domain $F$ convey information about the dimensionality of $F$. Roughly speaking, a generic continuous function should have large level sets
if $F$ itself is dense in the everyday sense, and small otherwise. This intuition prompts the investigation of the size of such level sets.

In \cite{BK}, Hausdorff dimension of level sets of the generic continuous function defined on $[0, 1]^p$ was determined. This result was vastly generalized in \cite{BBEtoph}: 
the concept of topological Hausdorff dimension was introduced, and
it was shown that the size of the level sets is governed by the so-called topological Hausdorff dimension of the domain. In \cite{BBElevel}, the connection 
between the topological Hausdorff dimension and the level sets of generic continuous functions on fractals was investigated more precisely. We note that topological Hausdorff dimension
sparked interest on its own right: in \cite{MaZha}  the authors studied topological Hausdorff dimension of fractal squares, and it was also mentioned and used in Physics papers, see for example
\cite{Balankintoph}, \cite{Balankinfracspace}, \cite{Balankintransport}, \cite{Balankintoph2}, and  \cite{Balankinfluid}.

A shortcoming of continuous functions is that they are determined by solely the topology of $F$. This independence from the metric structure is not inherited by the Hausdorff dimension of the level sets, 
as the Hausdorff dimension is an inherently metric concept. Nevertheless, it is a natural expectation that the investigation of other classes of functions might be more meaningful in this context,
and provide more information about the domain. To this end, we would like to rule out the possibility of compressing level sets arbitrarily well to the ``thinnest'' regions of the domain.
The simplest and most natural way to achieve this goal is considering H\"older functions instead of arbitrary continuous functions.
More specifically, we investigate the level sets of generic 1-H\"older-$\aaa$ functions, where genericity is understood in the Baire category sense in the topology given by the supremum norm.

We started to deal with this question in \cite{sier}. We defined $D_{*}(\aaa, F)$ as the essential supremum 
of the Hausdorff dimensions of the level sets of a generic
$1$-H\"older-$\aaa$ function defined on $F$, and verified that this quantity is well-defined for compact $F$.
We also showed that if $\alpha\in(0,1)$, then for connected self{-}{}similar sets, like the Sierpi\'nski triangle
$D_{*}(\aaa, F)$ equals  the Hausdorff dimension 
of almost every level-set in the range of  a generic 1-H\"older-$\aaa$ function, that is, 
one can talk about the Lebesgue-typical dimension of the generic 1-H\"older-$\aaa$ function.
We extended the notation $D_{*}(\aaa, F)$ to $\aaa = 0$ by taking the Lebesgue-typical dimension of the generic continuous function.
In our second paper \cite{sierc}, we provided results of a more quantitative nature, giving lower and upper estimates for $D_{*}(\aaa, \Delta)$ for
the Sierpi\'nski triangle $\Delta$.

In this paper, our goal is to define analogous quantities for the generic lower and upper box dimension of level sets. As usual,
lower box dimension shows strong analogy with Hausdorff dimension, and most of our previous results about the Hausdorff dimension of generic level sets apply to
the lower box dimension of generic level sets: the behaviour of any of them is determined by how well
the majority of the level sets passing through the domain can be squeezed into the thinnest parts of the fractal.
While their theory might diverge, it seems to be very difficult to prove different estimates about them on nice domains, such as self-similar sets satisfying the open set condition.
We will explain this intuition in more detail in Section \ref{sec:box}. Therefore, we will not put much emphasis on the investigation of lower box dimension.
Instead, we will be primarily concerned with upper box dimension, which behaves quite differently: 
taking $\limsup$ instead of $\liminf$ means that instead of searching the "narrowest" parts of the fractal this time we are interested in how much level sets and "level regions" 
can be spread over the fractal.  
This dichotomy results in vastly different quantities: as one increases $\alpha$, in terms of Hausdorff dimension and lower box dimension, generic level sets get larger
as the loss of compressibility reduces the ability to squeeze level sets to 
narrow "cross section" regions of the fractal. However, in terms of upper box dimension generic level sets get "smaller" for higher values of $\alpha$.
This time due to the fact that the "generic" Hölder-$\alpha$ functions can oscillate less rapidly and this way generate less many "level regions".
In a vague heuristic way we can compare the situation to the coarea formula 
\cite{Federer} which roughly states that for a Lipschitz function $f$ defined on an open set $G\sse \R^n$ one has $\ds \int_{G}||\gr f(x)||dx=\int_{-\oo}^{+\oo}
\cah^{n-1}(f^{-1}(r))dr$, that is if the gradient is large, then the integral of the 
$n-1$-dimensional Hausdorff measure of the level sets is also large, that is,
if the function is changing rapidly then we have "large" level sets.
In our paper we work with much less regular functions than Lipschitz functions, and our domains are more complicated than an open set. But the heuristic 
still works: the smaller $\aaa$, the wilder the generic function can oscillate, the larger the ``gradient" and the ``size" of the
generic level set is also larger.

The relevant notation and certain preliminary results are introduced in Section \ref{sec:prelimin}.

In Section \ref{sec:box}, we prove qualitative results concerning the lower and upper box dimension of generic level sets. 

In Section \ref{sec:sier_triangle_box}, these results are used for giving estimates concerning the upper box dimension of
level sets of Hölder functions defined on the Sierpiński triangle.

The paper is concluded by Section \ref{sec:open_problems}, in which we present a collection of open problems.


\section{Notation and preliminaries} \label{sec:prelimin}

 The interior of a set $H$ is denoted by  $\inte{H}$. The $t$-dimensional Hausdorff measure is denoted by $\mathcal{H}^t$. 

Following \cite{sier}, we introduce the following notation: if $(M,\mathpzc{d})$ 
is a metric space, then let $C(M)$ be the space of continuous functions  from $M$ to $\R$.   
For every $\alpha\in(0,1]$ and $c>0$ set
$$
C_{c}^{\alpha}(M) := \{ f\in C(M) : \text{for every $a,b\in M$ we have $|f(a)-f(b)|\le c(\mathpzc{d}(a,b))^\alpha$}\}
$$
and 
$$
C_{c^-}^{\alpha}(M) := \Union_{c'\in [0,c)} C_{c',\alpha}(M).
$$

Moreover, let $D^f(r,F)=D^f(r)=\dim_H(f^{-1}(r))$ for any function $f: F\to \mathbb{R}$,
that is $D^f(r)$ denotes the Hausdorff dimension of the function $f$ at level $r$.
We put 
\begin{displaymath}
D_{*}^f(F)=\sup\{d: \lambda\{r : D^f(r,F)\geq{d}\}>0\},
\end{displaymath}
where $\lambda$ denotes the one-dimensional Lebesgue measure, that is $D_{*}^f(F)$ is the essential supremum of Hausdorff dimension of level sets.
We denote by $\mg_{1,\aaa}(F)$, or by simply $\mg_{1,\aaa}$ the  set  of dense $G_{\ddd}$
sets in $C_{1}^{\aaa}(F)$.

We put
\begin{equation}\label{*defDcsaF}
D_{*}(\aaa,F)=\sup_{\cag\in \mg_{1,\aaa}}\inf\{ D_{*}^{f}(F):f\in \cag \}.
\end{equation}

This quantity was the main object of interest in \cite{sier} and \cite{sierc}. In \cite{sier}, the following simplification was proved.


\begin{theorem} \label{*thmgenex}
    If   $0< \aaa\leq 1$  and $F\subset\R^p$ is compact, then there is a dense $G_\delta$ subset $\cag$ of $C_1^\alpha(F)$ 
    such that for every $f\in\cag$ we have $D_*^f(F) = D_*(\alpha,F)$.
\end{theorem}

This result relied on the following lemma:


\begin{lemma}\label{*lemdfb}
    Suppose that   $0< \aaa\leq 1$,   $F\subset\R^p$ is compact, $E\subset\R^p$ is open or closed, and $\cau\subset C_1^\alpha(F)$ is open.
    If $\{f_1,f_2,\ldots\}$ is a countable dense subset of $\cau$, then there is a dense $G_\delta$ subset $\cag$ of $\cau$ such that 
    \begin{equation}\label{eqdfb_hausdorff}
    \sup_{f\in\cag} D_*^f(F\cap E) \le \sup_{k\in\N} D_*^{f_k}(F\cap E).
    \end{equation}
\end{lemma}

We extend the above notation to the lower and upper box dimension of level sets. Notably, we put 
$D_{\underline{B}}^f(r,F) = \underline{\dim}_{B}(f^{-1}(r))$ and $D_{\overline{B}}^f(r,F) = \overline{\dim}_{B}(f^{-1}(r))$.
As it will turn out, to have natural analogues to the machinery above, in case of upper box dimension it makes more sense to swap $\inf$ and $\sup$ in the definition. Consequently, we define
\begin{displaymath}
    D_{\underline{B}*}^f(F)=\sup\{d: \lambda\{r : D_{\underline{B}}^f(r,F)\geq{d}\}>0\},
\end{displaymath}
\begin{displaymath}
    D_{\overline{B}*}^f(F)= \begin{cases}
        \inf\{d: \lambda\{r : r\in f(F) \text{ and } D_{\overline{B}}^f(r,F)\leq{d}\}>0\}, & \text{if } \lambda(f(F))>0, \\
        0, & \text{if } \lambda(f(F))=0,
    \end{cases}
\end{displaymath}
that is while $D_{\underline{B}*}^f(F)$ is the essential supremum of the lower box dimension of level sets, $D_{\overline{B}*}^f(F)$ is the essential infimum
of the upper box dimension of nonempty level sets. (The nonempty clause was omittable previously, but is vital in this case.)
Moreover, let
\begin{equation}\label{def:generic_lower}
    D_{\underline{B}*}(\aaa,F)=\sup_{\cag\in \mg_{1,\aaa}}\inf\{ D_{\underline{B}*}^{f}(F):f\in \cag \} \text{ and}
\end{equation}
\begin{equation}\label{def:generic_upper}
   D_{\overline{B}*}(\aaa,F)=\inf_{\cag\in \mg_{1,\aaa}}\sup\{ D_{\overline{B}*}^{f}(F):f\in \cag \}.
\end{equation}
In Section \ref{sec:box}, we are going to analyze analogous statements to Lemma \ref{*lemdfb} and Theorem \ref{*thmgenex} about 
$D_{\underline{B}*}$ and $D_{\overline{B}*}$.

We recall the following extension theorem, which is a consequence of Theorem 1 of \cite{[GrunbHolderext]}:

\begin{theorem}\label{thm:Grunb}
Suppose that $F\sse \R^{p}$ and $f:F\to \R$ is a $c$-H\"older-$\alpha$ function.
Then there exists a $c$-H\"older-$\alpha$ function $g:\R^{p}\to \R$
 such that
 $g(x)=f(x)$ for $x\in F$.
\end{theorem}

We will use a well-applicable approximation result (Lemma 4.4 of \cite{sier}). If $F$ is compact, we say that a function $f:F\to\mathbb{R}$ is piecewise affine
if there is a finite set $\mathcal{S}$ of non-overlapping, non-degenerate simplices such that for any $S\in\mathcal{S}$, $f$ coincides with an
affine function on $S\cap F$.


\begin{lemma} \label{lemma:piecewise_affine_approx}
Assume that $F$ is compact and $c > 0$ is fixed. 
Then locally non-constant, piecewise affine $c$-Hölder-$\alpha$ functions defined on $F$ form a dense subset $C_{c^{-}}^{\alpha}(F)$.
\end{lemma}


We will also use a slicing theorem concerning Hausdorff dimension, which is the natural generalization of Marstrand's classical slicing theorem. The statement requires some technicalities.

We say that $A$ is a $t$-set if $\dim_H A = t$ and
$0<\mathcal{H}^t(A)<\infty$. By $G(p, k)$, we denote the Grassmannian of $k$-dimensional linear subspaces of $\mathbb{R}^p$, and by $\gamma_{p,k}$ its canonical measure, 
being the push-forward of the Haar measure on the group of $p$-dimensional invertible linear transformations. This concept is not straightforward to grasp, but we will be concerned exclusively with $k=p-1$, 
in which case it is easily accessible: by taking orthocomplementers, one can see that $G(p, p-1)$ is isomorphic to $G(p, 1)$, which is equivalent to the projective space $\mathbb{P}\mathbb{R}^{p-1}$. 
However, it is a factor of the sphere $\mathbb{S}^{p-1}$, and the measure $\gamma_{p, p-1}$ is simply equivalent to the push-forward of the $p-1$ dimensional spherical measure.

If $W\in G(p, k)$, it can be thought of as a subset of $\mathbb{R}^p$, hence the notation $W+a$ makes sense for $a\in \mathbb{R}^p$. The slicing theorem is the following (Theorem 10.10 of \cite{mattila_1995})

\begin{theorem} \label{thm:slicing_thm}
    Let $m \leq t \leq p$, and let $A\subseteq \mathbb{R}^p$ be a Borel $t$-set. Then for $\gamma_{p, p-m}$ almost all $W\in G(p, p-m)$
    $$\mathcal{H}^m(\{a\in W^{\perp}: \dim_H (A\cap (W+a))=t-m\})>0.$$
\end{theorem}

We note that self-similar sets, which are attractors of iterated function systems satisfying the Open Set Condition, (OSC) are $t$-sets due to the classical theorem of Hutchinson
\cite{[Hut]}. We will call such sets simply
OSC self-similar sets.   For the iterated function system defined by $(\Psi_i)_{i=1}^{m}$
the bounded non-empty
   open set $V$ is {\it witnessing OSC} if $\bigcup_{i=1}^{m}\Psi_i(V)\subseteq V$, and
    the union is disjoint. 

Later on we will need a slightly stronger form of Theorem \ref{thm:slicing_thm},  to which end we introduce $Pr_{W^{\perp}}(A)$, the orthogonal projection of $A$ to $W^{\perp}$. Notably, 
we will need that in $Pr_{W^{\perp}}(A)$  almost every intersection has dimension $\dim_H A -1$, not merely of positive measure.  
This property requires some sort of homogeneity of $A$. For a similarity $\Psi$ acting on $\mathbb{R}^p$, there is a unique orthogonal linear transformation $R$, similarity ratio $\rho\in (0, +\infty)$, and
translation vector $b\in \mathbb{R}^p$ such that $\Psi(x) = \rho R(x) + b$. We refer to $R$ as the {\it orthogonal part} of $\Psi$, and a self-similar set has {\it finitely many directions}
if it is determined by similarities such that their orthogonal parts generate a finite subgroup of the orthogonal group $O(p)$.
We recall \cite[Corollary~1]{wenxi} ($\dim$ is either the Hausdorff, box, or packing dimension):

\begin{proposition} \label{prop:wenxi}
    Assume that $A$ is a self-similar set with finitely many directions.
    If $\dim_H A > m$, then for $\gamma_{p, p-m}$ almost all $W\in G(p, p-m)$ 
    $$\mathcal{H}^m(Pr_{W^{\perp}}(A)\setminus \{a\in W^{\perp}: \dim (A\cap (W+a))= \dim A  -m\}) = 0,$$
    that is standard dimension drop occurs on full measure.
\end{proposition}

\hspace*{-0cm}\begin{figure}[h] 
    \includegraphics[scale=1]{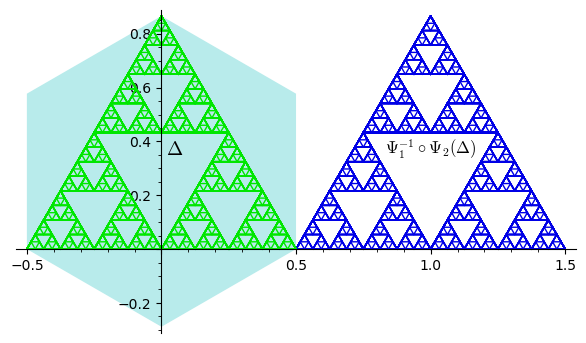}
 \caption{The Sierpiński triangle, $\Delta$ is in its light turquoise shaded central open set. To its right there is one of its neighbor sets, it is given by $\Psi_{1}^{-1}\circ \Psi_{2}(\Delta)$.
\label{fig:siercentral} }
\end{figure}

In Section \ref{sec:box}, it will be important to cleverly choose open sets witnessing that a given self-similar set is OSC. To this end, we will rely on \cite{10.2307/4097989}. 
Assume that an OSC self-similar
set $F$ is determined by the similarities $\Psi_1, ..., \Psi_m$. For ${\bf{i}} = (i_1, ..., i_k)$ and ${\bf{j}} = (j_1, ..., j_l)$, and $\Phi_{\bf{i}} = \Psi_{i_1}\circ ... \circ \Psi_{i_k}$, 
$\Phi_{\bf{j}} = \Psi_{j_1}\circ ... \circ \Psi_{j_k}$, we say that a mapping $\Phi_{\bf{i}}^{-1} \circ \Phi_{\bf{j}}$ is a \textit{neighbor map} if $i_1\neq j_1$, while the image of $F$ under
a neighbor map is a \textit{neighbor set} of $F$. 

Denote the set of neighbor maps by $\mathcal{N}$, and let 
\begin{equation}\label{*neighb}
H=\bigcup_{h\in \mathcal{N}}h(F).
\end{equation}
 Finally, let $V=\{x: d(x, F) < d(x, H)\}$ be the so-called \textit{central open set} of $F$, 
where $d$ denotes point-set distance.  For illustration, we refer to Figure \ref{fig:siercentral}, depicting the Sierpiński triangle determined by the similarities
$$\Psi_{1}\vekk{x}{y}=\matk{0.5}0 0 {0.5}\vekk{x}{y}+\vekk{-0.25}0,$$
$$\Psi_{2}\vekk{x}{y}=\matk{0.5}0 0 {0.5}\vekk{x}{y}+\vekk{0.25}0,$$
$$\Psi_{3}\vekk{x}{y}=\matk{0.5}0 0 {0.5}\vekk{x}{y}+\vekk{0}{\sqrt 3/4},$$
its central open set, and one of its neighbor sets. 

\cite[Theorem~1]{10.2307/4097989} essentially states the following:

\begin{theorem} \label{thm:central_open_set}
    The central open set witnesses OSC.
\end{theorem}

We recall some notation and results of \cite{sierc} concerning the Sierpiński triangle  $\Delta$. By definition, 
$\Delta=\bigcap_{n=0}^{\infty}\Delta_n$ where $\Delta_n$ 
is the union of the triangles appearing at the $n$th step of the construction. 
The set of triangles on the $n$th level is $\tau_n$, and $\tau:=\bigcup_{n\in\N\cup\{0\}} \tau_n$.
Usually the interior of the initial triangle $\Delta_{0}$ is used 
as the open set $V$ witnessing OSC for the Sierpiński triangle, but later the usage of the central open set will be more advantageous.
If $T\in\tau_n$ for some $n$ then we denote by $V(T)$ the set of its vertices.

We say that $f:\DDD\to \R$ is a {\it piecewise affine function at level $n\in \N$
on the Sierpi\'nski triangle} if it is affine on  any $T\in\tttt_{n}$.

If a piecewise affine function at level $n\in \N$ on the Sierpi\'nski triangle  satisfies the property that for any $T\in\tttt_{n}$ one can always find two vertices of $T$
where $f$ takes the same value, then we say that $f$ is a {\it standard piecewise affine function at level $n\in \N$
on the Sierpi\'nski triangle}.

A function $f:\DDD\to \R$ is a {\it strongly piecewise affine function
on the Sierpi\'nski triangle} if there is an $n\in \N$ such that it is a piecewise affine function at level $n$.

We will use \cite[Lemma~4.2]{sierc} in Section \ref{sec:sier_triangle_box}:

\begin{lemma}\label{lemma:standard_piecewise_affine}
Assume that $0<\alpha<1$, and $0<c$ are fixed. 
Then the locally non-constant standard strongly piecewise affine $c^{-}$-H\"older-$\alpha$ functions defined on $\DDD$ form a dense subset of the $c$-H\"older-$\alpha$ functions.
\end{lemma}

\section{Basic results for box dimension} \label{sec:box}

First, we state the direct analogues of Lemma \ref{*lemdfb} and Theorem \ref{*thmgenex} concerning lower box dimension.


\begin{lemma}\label{lemma:lower_box_generic}
    Suppose that   $0< \aaa\leq 1$,   $F\subset\R^p$ is compact, $E\subset\R^p$ is open or closed, 
    and $\cau\subset C_1^\alpha(F)$ is open.
    If $\{f_1,f_2,\ldots\}$ is a countable dense subset of $\cau$, then there is a dense $G_\delta$ subset $\cag$ of $\cau$ such that 
    \begin{equation}\label{eqdfb_lower_boxdim}
    \sup_{f\in\cag} D_{\underline{B}*}^f(F\cap E) \le \sup_{k\in\N} D_{\underline{B}*}^{f_k}(F\cap E).
    \end{equation}
\end{lemma}


\begin{theorem} \label{thm:lower_box_generic}
    If  $0< \aaa\leq 1$  and $F\subset\R^p$ is compact, then there is a dense $G_\delta$ subset $\cag$ of $C_1^\alpha(F)$ 
    such that for every $f\in\cag$ we have $D_{\underline{B}*}^f(F) = D_{\underline{B}*}(\alpha,F)$.   
\end{theorem}

\begin{proof}[Proofs of Lemma \ref{lemma:lower_box_generic} and Theorem \ref{thm:lower_box_generic}]
    Due to the highly similar nature of 
    Hausdorff dimension and
    lower box dimension one can repeat verbatim the proofs of Lemma \ref{*lemdfb} and Theorem \ref{*thmgenex}, presented in \cite{sier}. We omit the details.
\end{proof}

Hence the foundations of $D_{\underline{B}*}$ do not contain novelties, and as we preluded in the introduction, telling this quantity apart from $D_*$ seems to be difficult.
Notably, any lower estimate on $D_{*}$ also holds for $D_{\underline{B}*}$, and as the upper estimates for $D_{*}$ generally rely on Lemma \ref{*lemdfb}, 
using a dense set of functions with level sets having equal box dimension and Hausdorff dimension, Lemma \ref{lemma:lower_box_generic} yields the same upper estimates 
for $D_{\underline{B}*}$.
For example, \cite[Theorem~3.2]{sierc} and \cite[Theorem~4.5]{sierc} can be converted to a lower box dimension form by simply replacing $D_{*}$ by $D_{\underline{B}*}$:

\begin{theorem}\label{theorem:lower_box_sier}
    For any $0<\alpha<1$ and the Sierpiński triangle $\Delta$, we have 
    $$\frac{\frac{\alpha}{2}}{1+\frac{1+\log\frac{3}{\alpha}}{\log 2} + \frac{2}{\alpha}}\leq D_{\underline{B}*}(\alpha, \Delta)\leq 1-2^{-\alpha}.$$
\end{theorem}

On Figure \ref{fig:sier_upper_box_estimates} one can see the estimates
obtained from this theorem along with the ones obtained for
$D_{\overline{B}*}(\alpha, \Delta)$.
As far as upper estimates rely on dense sets of functions with nice level sets in the aforementioned sense, we will not have different estimates for $D_{\underline{B}*}$ and $D_*$. 
 However, so far
all of our constructions either in this paper, or in our earlier papers were tame in this sense, and if the domain $F$ is nice, for instance an OSC self-similar set,
it seems hard to produce functions which lack this property and for which one can still calculate dimensions precisely. This last requirement tempts one to use functions which have level sets
consisting of planar cross-sections of  $F$ with directions interacting well with the similarities determining $F$, so that they are easy to track.
For instance, in the case of the Sierpi\'nski triangle, both \cite[Theorem~4.5]{sierc} and Theorem \ref{thm:sier_upper_box} rely on such a construction, each level set
consisting of line segments parallel to the sides of the triangle
and one can simply count the triangles of $\tau_n$ intersecting a level set in order
to determine the dimension of the level set. However, this method yields coverings of uniform diameter, thus it is unable to tell apart Hausdorff and box-counting dimension by nature. 
It would be interesting to see examples which rely on a different approach.

 We conclude our observations concerning lower box dimensions by providing a compact example for which $D_{*}$ differs from $D_{\underline{B}*}$. Naturally, this fractal is not tame in the above sense,
as it has different Hausdorff and lower box dimension.

\begin{theorem}
    Let $F_0 = \{0\} \cup \left\{{1}/{n}: n=1,2,\ldots\right\}$, and let $F=F_0 \times [0,1]\subseteq \mathbb{R}^2$. Then for any $\alpha\in (0, 1)$, we have $D_{*}(\alpha, F) = 0$ , while
    $D_{\underline{B}*} = {1}/{2}$.
\end{theorem}

\begin{proof}
    We recall the well-known fact that $\dim_B F_0 = {1}/{2}$, (Example 3.5 of 
\cite{[Fa1]}).

    By applying Lemma \ref{lemma:piecewise_affine_approx}, we can find a countable dense subset $\{f_1, f_2, \ldots\}$ of $C_1^{\alpha}(F)$ such that each $f_k$ is
    piecewise affine, moreover we can also suppose that 
for each $k$, there is an $N_{k}$, such that 
\begin{equation}\label{*bound}
\text{$f_{k}(1/n, y)=
f_{k}(0, y)$ for any $n\geq N_k$ and $y\in [0,1]$.}
\end{equation}
Indeed, for arbitrary piecewise affine function $g\in C_1^{\alpha}(F)$, we can put 
\begin{displaymath}
    g_n(x, y)= \begin{cases}
        g(x, y) & \text{if } x>1/n, \\
        g(1/n, y), & \text{if } x\leq 1/n,
    \end{cases}
\end{displaymath}
which defines a sequence of piecewise affine functions which are also in $C_1^{\alpha}(F)$, and this sequence uniformly converges to $g$ due to $g$ being uniformly continuous.
 Notice that any $f_k$ restricted to any $\{x\}\times [0, 1]$ is a one-dimensional piecewise affine function of finitely many pieces, thus apart from finitely many values
    it admits each of its values finitely many times. (The exceptional values arise from the fact that $f_k$ might have constant pieces restricted to $\{x\}\times [0, 1]$.)
 From \eqref{*bound} it follows that altogether apart from finitely many values $f_k$ admits each of its values
   countably many times. Consequently, $D_*^{f_k}(F)=0$, thus Lemma \ref{*lemdfb} yields $D_{*}(\alpha, F) = 0$.

    Concerning $D_{\underline{B}*}$, we prove two inequalities. The upper bound follows by using the same dense set of functions $\{ f_{k} \}$ and Lemma \ref{lemma:lower_box_generic}. Notably, if we fix $k$,
    apart from finitely many exceptional $r$s, $f_k^{-1}(r) \cap (\{1/n\}\times [0, 1])$ for each $n\in\N$ consists of finitely many points 
and one can select a bound on the number of these points which is not depending on $n$.  Using \eqref{*bound} an easy calculation yields that 
    $$\underline{\dim}_B f_k^{-1}(r) \leq \dim_B F_0 = \frac{1}{2},$$ 
    and hence $D_{\underline{B}*} \leq {1}/{2}$.

    For the other inequality, notice that the generic $f$ is non-constant on $\{0\} \times [0, 1]$, that is, $f(\{0\} \times [0, 1])$ contains some interval $I$. By continuity, a bit shorter interval $J$
    is contained by any $f(\{x\} \times [0, 1])$ for small enough $x\in F_0$. Thus for any $r\in J$, we have that $f^{-1}(r)$ contains a point of $\{x\} \times [0, 1]$ for any
    $x \in F_0 \cap [0, x_0]$ for some $x_0>0$. Thus the projection $Pr_x f^{-1}(r)$ contains $F_0 \cap [0, x_0]$, which has box dimension ${1}/{2}$. 
    Consequently, $D_{\underline{B}*} \geq {1}/{2}$.

\end{proof}
\xcx

From now on, we will be concerned with upper box dimension exclusively. To start our investigations, we will use the following definitions:

\begin{definition} \label{def:nicely_connected}
    We say that $(\mathcal{T}_n)_{n=1}^{\infty}$ is a  {\it box dimension defining sequence} (BDDS) of $F\subseteq \mathbb{R}^p$ 
    if 
    \begin{itemize}
        \item  each $\mathcal{T}_n$ is a countable family of closed subsets covering  $F$ 
        in a non-overlapping manner,
        \item $\mathcal{T}_{n+1}$ is a refinement of $\mathcal{T}_{n}$ for each $n\geq 1$, that is for any $T\in\mathcal{T}_n$, 
        we have $T\cap F=\bigcup_{i=1}^{m}T_i \cap F$ for some $T_i\in \mathcal{T}_{n+1}$ ($i=1, 2, ..., m$),
        \item there exists $C>0$ and $0<q<1$ such that for any $n$ and $H\in \mathcal{T}_n$,
        $H$ contains a ball with diameter $C^{-1}q^{n}$ and $\diam(H) \leq Cq^{n}$. 
         The quantity  $q$ is the {\it base} of $(\mathcal{T}_n)_{n=1}^{\infty}$. 
    \end{itemize}
    
    The set $F\subseteq \mathbb{R}^p$ has {\it nice connection type}, if  it has BDDS  
    such that $F\cap H$ has countably many components for any $H\in \bigcup_{n=1}^{\infty}\mathcal{T}_n$.
\end{definition}

The philosophy behind the term "box dimension defining sequence" is that using such a sequence, one can calculate  the lower and upper
box dimension of any bounded set $F$.  
Notably, if for a set $X$ we put $A_n(X) = \{H\in\mathcal{T}_n: H\cap X \neq \emptyset\}$, and $a_n(X)$ is the
cardinality of $A_n(X)$
then $\liminf \frac{\log a_n(F)}{-n\log q} = \underline{\dim}_B(F)$ and $\limsup \frac{\log a_n(F)}{-n\log q} = \overline{\dim}_B(F)$.
 (Notice that any $\mathcal{T}_n$ is locally finite as each of its members contains a ball with radius bounded away from 0, hence
$a_n(F)<\infty$ for bounded sets.) 

The most standard example for such a sequence is given by $\mathcal{T}_n$ being the set of lattice cubes of edge length $2^{-n}$.
A similar example is obtained in $\mathbb{R}^2$ by putting $\mathcal{T}_n$ to be the set of lattice regular triangles of edge length
$2^{-n}$. The usage of these BDDSs immediately yields that the Sierpiński triangle and the Sierpiński carpet have
nice connection types. 

 While it is tempting to think that every connected set has nice connection type, it is not the case. For instance, this assumption fails for
$$C^{*} = \{(x, y): 0\leq x \leq 1, y\in x\cdot C\},$$
where $C$ is the triadic Cantor set, that is $C^*$ is the union of continuum many different line segments starting from the origin.
Nevertheless, additional constraints might enforce the presence of nice connection type, for instance:

\begin{lemma}\label{lemma:connected_is_nice_connected_self_similar}
    Assume that $F\subseteq \mathbb{R}^p$ is a connected, OSC self-similar set. Then $F$ has nice connection type.
\end{lemma}

\begin{proof}
    Denote the associated iterated function system by $(\Psi_i)_{i=1}^{m}$, and their respective similarity ratios by $(q_i)_{i=1}^m$.  Moreover, let $q = \min_{1\leq i \leq m}q_i$.
    Due to OSC, we can find some non-empty bounded open set $V$ such that $\bigcup_{i=1}^{m}\Psi_i(V)\subseteq V$, and
    the union is disjoint. In light of Theorem \ref{thm:central_open_set}, let $V$ be the central open set of $F$.
    
    Choose $C$ such that $\diam(V)$ and $q^{-1}$ are both less than $C$, furthermore $V$ contains a ball of diameter $C^{-1}$.
    Denote the closure of $V$ by $F'$. Then we have that the sets $\Psi_i(F')$ are non-overlapping, and $\bigcup_{i=1}^{m}\Psi_i(F')\subseteq F'$. Thus $F'\supseteq F$.
    Now let $\mathcal{T}$ consist of all the sets which can be obtained from $F'$ by applying a finite composition consisting of maps in $\{\Psi_i\}_{i=1}^{m}$. 
    We are going to sort elements of $\mathcal{T}$ to preliminary sets $(\mathcal{T}_n^0)_{n=0}^{\infty}$. Notably, for $T\in \mathcal{T}$, let $T\in \mathcal{T}_n^0$ if $n$ is the smallest
    integer for which $\diam(T)\leq Cq^n$. Now clearly each $\mathcal{T}_n^0$ covers $F$, and as the sets $\Psi_i(F')$ were non-overlapping, any overlapping of elements in some $\mathcal{T}_n^0$
    arises from containments. Thus the overlappings can be eliminated as follows: we obtain $\mathcal{T}_n$ from $\mathcal{T}_n^0$ so that any growing chain of sets in $\mathcal{T}_n^0$ is replaced by its maximal element. These families satisfy the conditions
    concerning BDDS: 
    \begin{itemize}
        \item any $\mathcal{T}_n$ is non-overlapping as any intersections are filtered by the final step, and it covers $F$ as $\mathcal{T}_n^0$ did so and the filtering step does not reduce the union.
        \item The refinement property is trivial as well by construction.
        \item Finally, the last condition on the existence of constants $C>0$ and $0<q<1$  is satisfied by our choice of $C$ and $q$: the one concerning the diameter holds
            by the definition of $\mathcal{T}_n^0$, while the one on the contained ball follows from the fact that the "parent" $T'$ of $T$ is a similar copy of $F'$ 
            with diameter exceeding $Cq^{n}\geq q^{n-1}$. Thus
            $T$ is a similar copy of $F'$ with diameter exceeding $q^{n}$, hence it contains a ball of diameter $C^{-1}q^{n}$. 
    \end{itemize}

    We state that $\mathcal{T}_n$ is a BDDS which witnesses that $F$ has nice connection type. To this end, we verify that for any sequence $\Psi_{i_1}, ..., \Psi_{i_k}$, and their composition $\Phi$,
    we have that $\Phi(F')\cap F = \Phi(F)$. It is sufficient as $\Phi$ is a similarity and $F$ is connected, hence $\Phi(F')\cap F$ is connected as well. Thus every set in some $\mathcal{T}_n^0$ intersects
    $F$ so that the intersection has finitely many components, and in particular, the same holds for sets in $\mathcal{T}_n$. As $F'\supseteq F$, one of the containments is obvious, and it suffices
    to verify $\Phi(F')\cap F \subseteq \Phi(F)$.
    
    We apply induction on the composition length. For the empty composition, the claim is $F'\supseteq F$, which holds, since $F'$ is such that $\bigcup_{i=1}^{m}\Psi_i(F')\subseteq F'$. Concerning the
    inductive step, assume that the claim holds for composition length $k-1$, and let $\Phi = \Psi_{i_1}\circ ... \circ \Psi_{i_k}$. To simplify notation, 
    let $\Phi_0 = \Psi_{i_1}\circ ... \circ \Psi_{i_{k-1}}$. Due to the induction hypothesis,
    $$\Phi_0(F')\cap F = \Phi_0(F),$$
    and hence due to $\Psi_{i_k}(F')\subseteq F'$, we find
    $$\Phi(F')\cap F \subseteq \Phi_0(F) = \bigcup_{j=1}^{m} \Phi_0(\Psi_j(F))$$
    That is,
    \begin{equation} \label{eq:induction}
    \Phi(F')\cap F = \Phi(F')\cap \Phi_0(F).
    \end{equation}

    Notice that due to the choice of $V$, recalling notation \eqref{*neighb}
	we have that $F'\subseteq \{x: d(x, F)\leq d(x, H)\}$. Consequently, we have
    $$\Phi(F')\subseteq \{x: d(\Phi^{-1}(x), F)\leq d(\Phi^{-1}(x), H)\},$$
    which due to self-similarity, immediately yields 
    $$\Phi(F')\subseteq \{x: d(x, \Phi(F))\leq d(x, \Phi(H))\}.$$
    This set can be further increased by replacing $H$ with a smaller set. We utilize this observation by introducing $H_0 = \bigcup_{j\neq i_k}\Psi_{i_k}^{-1}\Psi_j(F)$, which yields
    $$\Phi(F')\subseteq \{x: d(x, \Phi(F))\leq d(x, \Phi(H_0))\} = \bigcap_{j\neq i_k}\{x: d(x, \Phi_0(\Psi_{i_k}(F)))\leq d(x, \Phi_0(\Psi_j(F)))\}.$$
    However, the rightmost set in this equality can contain $x\in \Phi_0(\Psi_{j}(F))$ for some $j\neq i_k$ only if $x\in \Phi_0(\Psi_{i_k}(F))$ as well. Consequently, 
    $$\Phi(F')\cap \Phi_0(F)\subseteq \Phi_0(\Psi_{i_k}(F)) = \Phi(F).$$
    Taking \eqref{eq:induction} into consideration, it concludes the proof.
\end{proof}


\begin{lemma}\label{lemma:upper_box_generic}
    Suppose that $0< \aaa\leq 1$,  $F\subset\R^p$ is compact with nice connection type, witnessed by $(\mathcal{T}_n)_{n=1}^{\infty}$.
    Moreover, assume that $E=F$ or $E = F \cap H$ for some $H\in \bigcup_{n=1}^{\infty}\mathcal{T}_n$, and $\cau\subset C_1^\alpha(F)$ is open.
    If $\{f_1,f_2,\ldots\}$ is a countable dense subset of $\cau$, then there is a dense $G_\delta$ subset $\cag$ of $\cau$ such that 
    \begin{equation}\label{eqdfb_upper_boxdim}
    \inf_{f\in\cag} D_{\overline{B}*}^f(E) \ge \inf_{k\in\N} D_{\overline{B}*}^{f_k}(E).
    \end{equation}
\end{lemma}

\begin{proof}[Proof of Lemma \ref{lemma:upper_box_generic}]
    We can assume $F=E$, as any other permitted $E$ is a compact set with nice connection type.

    Since countable union of sets of measure zero is still of measure zero we can choose
    a measurable set $R_{0}\sse \R$ such that $\lll(\R\sm R_{0})=0$ and  for any $k$ 
    \begin{equation}\label{*lrdfkr}
    \dub^{f_{k}}(r,E)\ge  \inf_{k'\in\N} \dub^{f_{k'}}(E) \text{ for any $r\in R_{0}$.}
    \end{equation}
By our assumption about  nice connection type one can choose countably many components of the form $H^j$ of $H\cap F$ for all $H\in \bigcup_{n=1}^{\infty}\mathcal{T}_n$. Hence we
    can also assume that $R_0$  does not contain
     $\min_{H^j}f_k$ and $\max_{H^j}f_k$ for any $k$ and 
component $H^{j}$ of the above form.    
    
    Suppose that $D_1<\inf_{k\in\N} \dub^{f_k}(E)$, and fix $k\in \N$ and $r\in R_{0}$.
    For every $n\in\N$ we can choose an $m_{n,k,r} \in \N\cap [n,\infty)$ such that $\frac{\log{a_{m_{n,k,r}}( f_k^{-1}(r)\cap E)
    }}{-m_{n,k,r}\log q}>D_1$,  where $q$ is the base of $(\mathcal{T}_n)_{n=1}^{\infty}$. 
	
    Next we suppose that apart from $r$ and $k$,  $n$ is also fixed.

  Since $r\in R_0$ is fixed, there is a radius $\rrr_{k,n,r}>0$ such that for any $g \in B(f_k, \rrr_{k, n, r})$, we have that
    $H\in A_{m_{n, k, r}}(f_k^{-1}(r)\cap E)$ yields $H\in A_{m_{n, k, r}}(g^{-1}(r)\cap E)$. 
    Indeed, if $H\in A_{m_{n, k, r}}(f_k^{-1}(r)\cap E)$, we can find a component $H^j$ such that $r \in \inte f_k(H^j)$, thus there exists
    a radius $\rrr_{k,n,r,H}>0$ such that $H\in A_{m_{n, k, r}}(g^{-1}(r)\cap E)$ for $g\in B(f_k, \rrr_{k, n, r,H})$.
    As the set 
    $A_{m_{n, k, r}}(f_k^{-1}(r)\cap E)$ is finite, $\rrr_{k, n, r} = \min_{H\in A_{m_{n, k, r}}(f_k^{-1}(r)\cap E)} \rrr_{k,n,r,H}$
    is a satisfying choice. Hence $a_{m_{n, k, r}}(f_k^{-1}(r)\cap E)\leq a_{m_{n, k, r}}(g^{-1}(r)\cap E)$.

    We presribe these radii explicitly to assure measurability. That is, let
    $$\rrr_{k,n,r,H} = \frac{1}{2}\sup_{j:r \in \inte f_k(H^j)}\min \{|r- \max_{x\in H^j} f_k(x)|, |r- \min_{x\in H^j} f_k(x)|, 1\},$$
    which is a continuous function of $r$ for $r\in f_k(H\cap F)\cap R_0$, as the innermost functions are continuous. One should note that these domains are not necessarily connected.
    However, they are measurable, thus these functions can be
    extended measurably to $f_k(F) \cap R_0$ by putting 1 where they are not defined yet. (As $f_k$ is continuous, $f_k(H\cap F)$ is the union of countably many intervals, hence measurable.)
    Then the definition of $\rrr_{k, n, r}$ becomes valid with $\rrr_{k, n, r} = \min_{H}\rrr_{k,n,r,H}$, which is
    a measurable function as the minimum of a countable set of measurable functions.

    By taking a suitable superlevel set of $\rrr_{k, n, r}$, we get a measurable set $R_{k, n}$ such that $\lambda(R_{k,n})>\lambda(f_k(F))-2^{-n}$, and 
    $\rrr_{k, n} :=\inf_{r\in R_{k, n}}\rrr_{k, n, r}>0$. By further reducing $\rrr_{k, n}$, we can assume that 
    for any $f \in B(f_k, \rrr_{k, n})$, we have that 
    $$\lambda(f(F) \setminus f_k(F)) < 2^{-n},$$
    and hence
    $$\lambda(f(F) \setminus R_{k,n}) < 2^{-(n-1)}.$$
    Moreover, for any $r\in R_{k,n}$ we have
    $$\frac{\log{a_{m_{n,k,r}}(f^{-1}(r)\cap E)
    }}{-m_{n,k,r}\log q}\geq \frac{\log{a_{m_{n,k,r}}(f_k^{-1}(r)\cap E)
    }}{-m_{n,k,r}\log q} \geq D_1.$$
    
    Now let $\cag_{n}=\bigcup_{k}B(f_{k},\rrr_{k,n})\cap \cau$, and
    $\cag =\bigcap_{n}\cag_{n}$.

    Suppose $f\in \cag.$ 
    Then there exists a sequence $k_{n}$ such that $f\in B(f_{k_{n}},\rrr_{k_{n},n})$ for every $n$. Consequently, 
    $$\lambda(f(F) \setminus R_{{k_n}, n}) < 2^{-(n-1)}.$$
    Thus by the Borel--Cantelli lemma, almost every $r\in f(F)$ is contained in infinitely many $R_{{k_n}, n}$. That is, 
    $$\frac{\log{a_{m_{n,k,r}}(f^{-1}(r)\cap E)}}{-m_{n,k,r}\log q}\geq D_1$$
    for infinitely many $m_{n,k,r}$. Hence $\dub^{f}(r,F)\geq D_1$.
    As $D_1<\inf_{k\in\N} \dub^{f_k}(E\cap F)$ was chosen arbitrarily, the lemma is proven.
    
\end{proof}

This lemma can be used to prove the following theorem:


\begin{theorem} \label{thm:upper_box_generic}
    Suppose that $0< \aaa\leq 1$,  $F\subset\R^p$ is compact with nice connection type, witnessed by $(\mathcal{T}_n)_{n=1}^{\infty}$.
    Then there is a dense $G_\delta$ subset $\cag$ of $C_1^\alpha(F)$ 
    such that for every $f\in\cag$ we have $D_{\overline{B}*}^f(F) = D_{\overline{B}*}(\alpha,F)$.   
\end{theorem}

\begin{proof}
    Let $D_0 := \dub(\alpha,F)$.

    For every $k\in\N$ choose a $\cag^k\in\mg_{1,\alpha}(F)$ for which $D_0+\frac1k\ge \sup_{f\in\cag^k}\dub^f(F)$. 
    Set $\cag_0=\bigcap_{k=1}^\infty \cag^k$. 
    We have that $\cag_0\in\mg_{1,\alpha}(F)$ and $D_0\ge \sup_{f\in\cag}\dub^f(F)$.

    It is enough to prove that for every $k\in\N$ there is a $\cag_k\in\mg_{1,\alpha}$ such that 
    \begin{equation}\label{dlb^f kicsi}
    \inf_{f\in\cag_k}\dub^f(F) \ge D_k :=D_0-\frac1k ,
    \end{equation}
    since then $\cag:=\bigcap_{k=0}^\infty \cag_k$ is a proper choice.
    
    Fix $k\in\N$. 

    The set $G_{k} := \{f\in\cag_0 : \dub^f(F)\ge D_k\}$ cannot be nowhere dense in $C_1^\alpha(F)$, since otherwise $\cag':=\cag_0\setminus \cl(G_{k})$ would be in $\mg_{1,\alpha}(F)$ and it would hold that 
    $$
    \sup_{f\in\cag'}\dub^f(F)  \le D_k < D_0 = \dub(\alpha,F),
    $$ 
    which contradicts the definition of $\dub(\alpha,F)$.
    Hence we can take $f_1\in C_1^\alpha(F)$ and $\delta_1>0$ such that $G_{k}$ is dense in $B(f_1,\delta_1)\cap C_1^\alpha(F)$.
    Choose a $\delta_2>0$ to satisfy $\delta_2^\alpha \le  {\delta_1}/{64}$.
    By definition  and the local finiteness of each $\mathcal{T}_n$, 
    we can take $H_1, ..., H_m \in \bigcup_{n=1}^{\infty}\mathcal{T}_n$ such that $\diam H_j < \delta_2$ for any $j$,
    and $\bigcup_{j=1}^{m}H_j \supseteq F$.
    
    Suppose that $j$ is fixed, $\eps>0$ and $g_0\in C_1^\alpha(F)$ is an arbitrary function.
    Let $E:=H_j\cap F$.
    By the H\"older property, for every $f\in C_1^\alpha(F)$
    $$
    \diam\big(f(E)\big)\le (2\delta_2)^\alpha \le \frac{\delta_1}{32}.
    $$
    Thus setting 
    $$
    g_1(x) := \min\big\{\max\{g_0(x)-g_0(a)+f_1(a),f_1(x)-\delta_1/2\},f_1(x)+\delta_1/2\big\}.
    $$
    we obtain 
    \begin{equation}\label{g_1}
    g_1|_E=g_0|_E-g_0(a)+f_1(a)
    \end{equation}
    (since $g_1(a)=f_1(a)$ and $\diam\big(g_0(E)\big)+\diam\big(f_1(E)\big)\le  {\delta_1}/{16}$).
    As $g_1\in B(f_1,\delta_1)$, 
     we can take $g_2\in G_k$ such that $\big|\big|g_1|_E-g_2|_E\big|\big|<\eps/100$.
    Set 
    $$
    g_3(x) := \min\big\{\max\{g_2(x)-g_2(a)+g_0(a),g_0(x)-\eps\},g_0(x)+\eps\big\}.
    $$
    Obviously $\big|\big|g_3-g_0\big|\big|\le\eps$.
    By \eqref{g_1}  and by the  definition of $g_2$, for every $x\in  E$
    \begin{equation}\label{g_1_cau-ban}
    \begin{gathered}
    \big|g_2(x)-g_2(a)+g_0(a)-g_0(x)\big| \\
    \le \big|g_2(x)-g_1(x)\big| + \big|g_1(a)-g_2(a)\big| + \big|g_0(a)-g_1(a) + g_1(x)-g_0(x)\big| \\
    \le \eps/100+\eps/100+0 <\eps,
    \end{gathered}
    \end{equation}
    hence $g_3(x) = g_2(x)-g_2(a)+g_0(a)$ for every $x\in  E$.
    Thus $\dub^{g_3}( E) = \dub^{g_2}( E) \ge D_k$ since $g_2\in G_k$.
    
    To sum up, for every $g_0\in C_{1}^{\alpha}(F)$, $1\leq j \leq m$ and $\eps>0$ we can find a $g_3\in C_{1}^{\alpha}(F)$ such that $||g_0-g_3||\le\eps$ and $\dub^{g_3}\big(H_j\cap F\big) \ge D_k$.
    Consequently, by Lemma \ref{*lemdfb} for every $1\leq j \leq m$ there is a $\cag_j^k\in\mg_{1,\alpha}$ satisfying 
    $$\sup_{f\in\cag_j^k} \dub^f(H_j \cap F) \ge D_k.$$
    Then \eqref{dlb^f kicsi} is true for $\cag^k := \bigcap_{j=1}^{m} \cag_j^k$, which completes the proof.
\end{proof}

This theorem has exciting consequences, as it can be the starting point of an industry. Notably, due to the fact that a dense set of
functions can be used to give a lower bound for $D_{\overline{B}*}(\alpha, F)$, while it yields an upper bound for
$D_{\underline{B}*}(\alpha, F)$ or $D_{*}(\alpha, F)$, dense sets of functions which seemed to be the most useless previously for estimating
$D_{*}(\alpha, F)$ from above, give the best lower estimates for $D_{\overline{B}*}(\alpha, F)$, given that the level sets for this
dense set of functions satisfy that they have box dimension equal to their Hausdorff dimension. As already noted, experience shows that 
standard, conveniently definable constructions tend to have this favorable property.
In our study of the Sierpiński triangle below, we will
showcase this industry. 

It is an interesting question whether one can drop or weaken the condition on having nice connection type in Lemma \ref{lemma:upper_box_generic} 
and subsequently, in Theorem \ref{thm:upper_box_generic}. We have yet to give a complete answer to this question, but sketch an example 
which demonstrates that this condition cannot be dropped entirely. Notably, consider a self-similar set $F$ determined by certain subsquares
of side length $2^{-n}$ in $[0, 1]^2$, so that it satisfies the strong separation condition and every row contains at least two subsquares.
(For $n\geq 3$, it is possible.) Then one can give a dense system of functions $\{f_1, f_2, ...\}$ in $C_1^{\alpha}(F)$ for $0<\alpha<1$
such that each $f_i$ is
piecewise constant. This gives rise to a dense
 $G_\delta$ set $\mathcal{G}_1$ of functions such that for any $f\in \mathcal{G}_1$, $\lambda(f(F))=0$,
hence even $D_{\overline{B}*}(\alpha, F)=0$. However, a slight perturbation of $\{f_1, f_2, ...\}$ results in a dense system of functions
$\{g_1, g_2, ...\}$ in $C_1^{\alpha}(F)$, which is not piecewise constant, but consists of "tilted" pieces, that is the pieces of
$g_j$ are of the form $g_j(x, y)=c_j y + d_j$. Due to the definition of $F$, the image of $g_j$ has positive measure, and
every nonempty level set has dimension at least $\frac{1}{n}$. This dense system of functions shows that Lemma \ref{lemma:upper_box_generic}
fails for $F$.

 Now we prove a simple lemma which can be used to give $\alpha$-dependent upper bounds on $D_{\overline{B}*}$. 

\begin{lemma} \label{lemma:from_covering_to_box}
    Assume that $0<\alpha\leq 1$, and $F\subseteq \mathbb{R}^p$ has coverings $(\mathcal{S}_n)_{n=1}^{\infty}$, 
    satisfying the following properties
    for some constants $C,l,\rho>1$:
    \begin{itemize}
        \item the cardinality of $\mathcal{S}_n$ is at most $Cl^n$ for some $C,l>1$,
        \item if $S\in \mathcal{S}_n$, then $\diam(S)\leq C\rho^{-n}$.
    \end{itemize}
    Then for any $f\in C_{\alpha}^{1}(F)$ and almost every $r\in \mathbb{R}$, we have
    $\overline{\dim}_B(f^{-1}(r))\leq \frac{\log l}{\log \rho} - \alpha$. In particular,
    $D_{\overline{B}*}(\alpha, F)\leq \frac{\log l}{\log \rho} - \alpha$.
\end{lemma}

\begin{proof}
    Fix $f\in C_{\alpha}^{1}(F)$. Moreover,
    fix $d>0$, and let $r\in R_n$ if $f^{-1}(r)$ intersects at least $\rho^{nd}$ elements of $\mathcal{S}_n$.
    Applying Fubini's theorem to the set $\{(T, r): f^{-1}(r)\cap T \neq \emptyset\}$ in the product space $\tau_n \times \mathbb{R}$,
    where the counting measure is considered on $\mathcal{S}_n$, we obtain
    $$\lambda(R_n) \rho^{nd} \leq C^{1+\alpha} l^n \rho^{-n\alpha},$$
    due to $\mathcal{S}_n$ having $l^n$ elements with diameter at most $C\rho^{-n}$, each of them having image of diameter at most 
    $C^{\alpha}\rho^{-n\alpha}$.
    Consequently,
    $$\lambda(R_n) \leq C^{1+\alpha}\left(\rho^{\frac{\log l}{\log \rho} - d -\alpha}\right)^n=C^{1+\alpha}c^n$$
    for some $c>0$. If $c<1$, then $\sum_{n=1}^{\infty} \lambda(R_n)$ is
    convergent, hence due to the Borel--Cantelli lemma almost every $r$ is contained by only finitely many $R_n$. 
    Consequently, using the cover provided by $\mathcal{S}_n$, we find
    $$\overline{\dim}_B(f^{-1}(r)) \leq \frac{\log \rho^{nd}}{\log C\rho^{n}}=d$$
    for almost every $r$. However, $c<1$ is equivalent to $d> \frac{\log l}{\log \rho} - \alpha$, thus we eventually find
    $$\overline{\dim}_B(f^{-1}(r)) \leq \frac{\log l}{\log \rho} - \alpha$$
    for almost every $r$.
\end{proof}

This theorem can be used naturally for self-similar sets, we will apply it in Section \ref{sec:sier_triangle_box} 
for the Sierpi\'nski triangle. For another example, it also yields the upper bound $D_{\overline{B}*}(\alpha, [0,1]^p)=p-\alpha$.

Our next goal is to provide a general lower bound of $D_{\overline{B}*}(\alpha, F)$. The intuition is the following: 
According to Lemma \ref{lemma:piecewise_affine_approx}, for $F$ compact and $0<\alpha<1$, if $0<c$ is fixed, then locally non-constant piecewise
affine $c^{-}$-Hölder-$\alpha$ functions form a dense subset of $c$-Hölder-$\alpha$ functions.
We would like to apply Theorem \ref{thm:upper_box_generic} to this dense set, thus we will require that $F$ has nice connection type.
If $f$ is a piecewise affine function, its level sets can be written as the intersection of $F$ with pieces of hyperplanes.
However, according to Theorem \ref{thm:slicing_thm}, typically, positively many such intersections have dimension $\dim_H F -1$, which hints that $D_{\overline{B}*}(\alpha, F) \geq \dim_H F -1$.
To make this argument precise, we have to deal with the following difficulties: on one hand, we need to ensure that we actually consider "good" hyperplanes, not being exceptional in Theorem \ref{thm:slicing_thm} 
applied to $F\cap S$. It can be done quite generally with a perturbation argument to be described in the proof of Theorem \ref{thm:gen_lower_bound_box}.
On the other hand, we need that almost every intersection has dimension $\dim_H F -1$, not merely of positive measure. As we already discussed in Section \ref{sec:prelimin}, 
this requires some homogeneity property of $F$. It seems to be a difficult question what is the weakest condition one has to impose. We will rely on Proposition \ref{prop:wenxi} (and assume its conditions)
to prove our theorem, providing a somewhat general lower bound on $D_{\overline{B}*}(\alpha, F)$.

\begin{theorem} \label{thm:gen_lower_bound_box}
    Assume that $F\subseteq \mathbb{R}^p$ is a nicely connected self-similar set with finitely many directions such that $\dim_H F = s$.
    Then $0<\alpha<1$ implies $D_{\overline{B}*}(\alpha, F) \geq \max(0, s-1)$.
    In particular, this bound holds for connected OSC self-similar sets with finitely many directions.
\end{theorem}

\begin{proof}
 
    If $s\leq 1$, there is nothing to prove, hence assume $s>1$.

 We can apply Theorem \ref{thm:upper_box_generic}: it suffices to construct a dense subset of $C_1^{\alpha}(F)$ such that for any
    $f$ coming from this set, almost every nonempty level set of $f$ has  Hausdorff  dimension at least $s-1$. 
    We apply Lemma \ref{lemma:piecewise_affine_approx} with $c=1$ to produce a dense subset in $C_1^{\alpha}(F)$ of piecewise affine functions, denoted by $PWA(F)$.

    We say that $f\in PWA(F)$ belongs to $PWANB(F)$, if each $S\in\mathcal{S}$ is {\it non-boundary} for $F$, that is for any $S\in\mathcal{S}$, either $\inte S \cap F \neq \emptyset$, or $S\cap F = \emptyset$.
    We claim that $PWANB(F)$ is still dense. 
    To prove this claim,  fix $f\in PWA(F)$, and  using Theorem \ref{thm:Grunb} extend it to $\mathbb{R}^p$ such that it remains $1^{-}$-Hölder-$\alpha$. 
   Next we show that one can select an arbitrarily small $c\in\mathbb{R}^p$, such that the function $\tilde{f}(x) = f(x+c)$ is in $PWANB(F)$, that is all the simplices $\tilde{\mathcal{S}}= \bigcup_{S\in\mathcal{S}}\{S-c\}$ are non-boundary for $F$.
    As $f$ is uniformly continuous, the existence of arbitrarily small $c$s would imply the denseness of $PWANB(F)$ due to the denseness of $PWA(F)$.

    Fix $S\in \mathcal{S}$ and denote by $A_S$ the set of $c$s for which $S-c$ is a non-boundary simplex. Then $A_S$ is clearly open. Moreover, if $c\notin A_S$, then $S-c$ contains some $x\in F$ on its boundary,
    which yields the existence of some $c'\in A_S$ arbitrarily close to $c$, ($c'=c+v$ is satisfying for any $v$ which points into $S$ from $x$). Thus $A_S$ is dense. However, then $A=\bigcap_{S\in \mathcal{S}}A_S$
    is still dense and open, and using any $c\in A$ yields $\tilde{f} \in PWANB(F)$. Thus $PWANB(F)$ is dense indeed.

    Fix $f\in PWANB(F)$ with underlying simplices $\mathcal{S}$.  By omitting simplices not intersecting $F$, we can assume that $\inte S \cap F \neq \emptyset$
    for any $S\in\mathcal{S}$. Then due to self-similarity, we know that $\dim_H(F\cap S) = \dim_H(F) = s$. 
    
    According to Proposition \ref{prop:wenxi}, for any $S\in \mathcal{S}$, 
    we can find some $\mathcal{N}_S \subseteq Gr(p, p-1)$ of zero $\gamma_{p, p-1}$ measure
    such that for any $W \notin \mathcal{N}_S$, and almost every $a\in Pr_{W^{\perp}}(F\cap S)$, we have that 
    $\dim_H ((F\cap S)\cap (W+a))=s-1$.
    
    We will consider the perturbation 
    $\tilde{f}(x)=f(x) + \langle t, x\rangle$ for fixed $t\in \mathbb{R}^p$, where $\langle \cdot, \cdot \rangle$ denotes inner/scalar product. 
    The resulting $\tilde{f}$ is also a piecewise affine function, and for any $S$ and $r\in f(F\cap S)$, 
    $\tilde{f}^{-1}(r)\cap S = (W_t + \tau) \cap F\cap S$ for some $\tau \in Pr_{W^{\perp}}(F\cap S)$ and $W_t\in Gr(p, p-1)$ depending on $t$.
    We say that $t$ is $S$-bad, if $W_t\in \mathcal{N}_S$. We claim that the set of $S$-bad $t$s is of zero Lebesgue measure.

    We can determine $W_t$ explicitly: if $f(x)=\langle \alpha_S, x\rangle + c_S$ on $S$
    for some $\alpha_S \in \mathbb{R}^p$ and $c_S\in \mathbb{R}$, then $W_t$ is the hyperplane orthogonal to $\alpha_S + t$. Thus 
    $\alpha_S + t$ is $S$-bad if and only if $\frac{\alpha_S + t}{\|\alpha_S + t\|}\in S^{n-1}$ is a unit normal vector of some
    $L\in \mathcal{N}_S$. Denote the set of such unit normal vectors by $E_{\mathcal{N}_S}$. As $\mathcal{N}_S$ has zero  $\gamma_{p, p-1}$ measure
    on $Gr(p, p-1)$, $E_{\mathcal{N}_S}\subseteq \mathbb{S}^{p-1}$ has zero spherical measure, due to the relationship between these measures explained before Theorem \ref{thm:slicing_thm}.
    That is, we
    need that $\frac{\alpha_S + t}{\|\alpha_S + t\|}$ evades the null-set $E_{\mathcal{N}_S}$ for almost every $t$. After embedding $\mathbb{S}^{p-1}$ into $\mathbb{R}^p$, we see that this
    evasion happens if and only if $\alpha_S + t$ evades the $p$-dimensional Lebesgue null-set $\bigcup_{\lambda>0}\lambda E_{\mathcal{N}_S} \subseteq \mathbb{R}^{p}$. (This set is a null-set due
    to Fubini's theorem.)
    That is, the set of $S$-bad $t$s is of zero Lebesgue measure, as we claimed. 
    Consequently, there exist $t$s with arbitrarily small norm which are
    not $S$-bad for any $S\in \mathcal{S}$. Choosing such a $t$, which is small enough, results in an $\tilde{f}$, which is 
    $1$-Hölder-$\alpha$, and for almost every level set intersecting some $S\in \mathcal{S}$, we have $\dim_H(f^{-1}(r))= \dim_H(F\cap S) -1 = s-1$.

    Finally, the observation concerning OSC self-similar sets follows from Lemma \ref{lemma:connected_is_nice_connected_self_similar}.

\end{proof}


\section{$D_{\overline{B}*}$ on the Sierpiński triangle} \label{sec:sier_triangle_box}

 In this section, we will give a lower and an upper estimate  of  $D_{\overline{B}*}(\alpha, \Delta)$ for the Sierpiński triangle $\Delta$.
Remarkably, in the limits $\alpha\to 0+$ and $\alpha\to 1-$, both of them coincide with the trivial upper bound $\overline{\dim}_{B}(\Delta) = \frac{\log 3}{\log 2}$ 
and with the general lower bound provided by Theorem \ref{thm:gen_lower_bound_box}, respectively.
These bounds are visualized on Figure \ref{fig:sier_upper_box_estimates}. 

\hspace*{-4cm}\begin{figure}[h] 
    \includegraphics[scale=0.35]{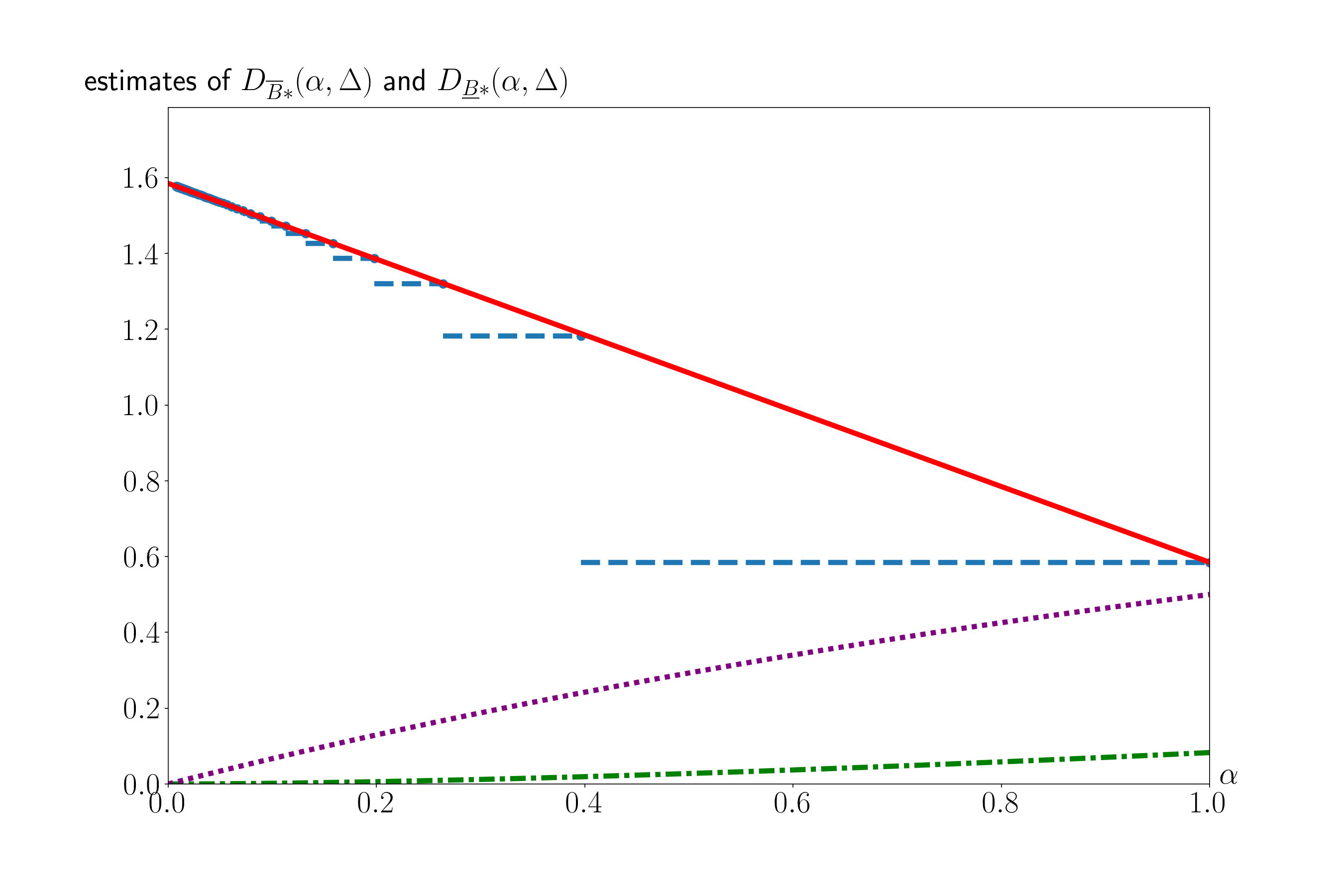}

    \caption{A visualization of the lower (blue, dashed line) and upper (red, simple line) bound for $D_{\overline{B}*}(\alpha)$,
    given by Theorems \ref{thm:sier_upper_box} and \ref{thm:lower_bound_upper_box_sier}, respectively.
     \label{fig:sier_upper_box_estimates} We note that the bounds only coincide in the limits $\alpha\to 0+$ and $\alpha\to 1-$. At other points, while the figure might
     hint otherwise, the blue and red contours are not touching. For more details, see Remark \ref{remark:bounds_close}. We also plotted
 the lower (green, dash-dot line) and upper estimates (purple, dotted line) of $D_{\underline{B}*}(\alpha)$ which are given by Theorem \ref{theorem:lower_box_sier}.}
\end{figure}

 First, we prove the upper bound, as it is a mere corollary at this point.

\begin{theorem}\label{thm:lower_bound_upper_box_sier}
    Assume $0<\alpha\leq 1$  Then for any $f\in C_{\alpha}^{1}(\Delta)$ and almost every $r\in \mathbb{R}$, we have
    $\overline{\dim}_B(f^{-1}(r))\leq \frac{\log 3}{\log 2} - \alpha$. In particular,
    $D_{\overline{B}*}(\alpha, \Delta)\leq \frac{\log 3}{\log 2} - \alpha$.
\end{theorem}

\begin{proof}
    The claim follows directly from Lemma \ref{lemma:from_covering_to_box}, applied to $F=\Delta$ and $\mathcal{S}_n=\tau_n$.
\end{proof}

Capitalizing on Lemma \ref{lemma:upper_box_generic}, a lower estimate can be achieved by constructing a dense sequence $f_1, f_2, ...$
of $C_1^\alpha(\Delta)$ such that for any $k$
we can determine $D_{\overline{B}}^{f_k}(\Delta)$, but in contrast to what happens in \cite{sierc} when one aims to estimate the Hausdorff dimension of the generic level sets, now
our goal is to set this quantity as large as possible.
We will construct a one-parameter family of functions $(\Phi_n)^{\infty}_{n= 2 }$, 
for which $d_n=D_{\overline{B}}^{\Phi_n}(\Delta)$ is large.  

For fixed $n$, the function $\Phi_n$ will be used as the building block 
of a dense set of functions so that for any element $f$ of it, we have $D_{\overline{B}}^{f}(\Delta)\geq d_n$.
However, $\Phi_n$ will have Hölder exponent $\alpha_n\to 0$, which will eventually yield better bounds for small $\alpha$s.

To construct $\Phi_n$, for the time being assume that the Sierpiński triangle is embedded into $\mathbb{R}^2$ so that its vertices are
$(0, 0), (2\sqrt 3, 0), (1/2, 1)$.  (We note that this choice of coordinates does not coincide with the one used in Figure \ref{fig:siercentral}.) 
As $n$ will be fixed during any argument, we suppress it in the notation: $\Phi=\Phi_n$.

Let $\mathcal{I}_{2n} = \{\frac{i}{2^{2n}}, \quad i = 0, 1, ..., 2^{2n}\}$ .
Any interval determined by adjacent points of $\mathcal{I}_{2n}$ is called a base interval of $\mathcal{I}_{2n}$.
We introduce an auxiliary function $\varphi=\varphi_n:[0,1]\to \mathbb{R}$. 
This function is defined to be piecewise linear with turning points in
$\mathcal{I}_{2n, 3} = \{\frac{i}{2^{2n}}, \quad 3|i, \quad i = 0, 1, ..., 2^{2n}\}\cup\{1\}$, such that it is constant 1 on the last interval,
and before that, it takes values 0 and 1 alternatingly.
(See Figure \ref{fig:aux_function}.)

\begin{figure}[h] 
    \centering
    \includegraphics[scale=0.3]{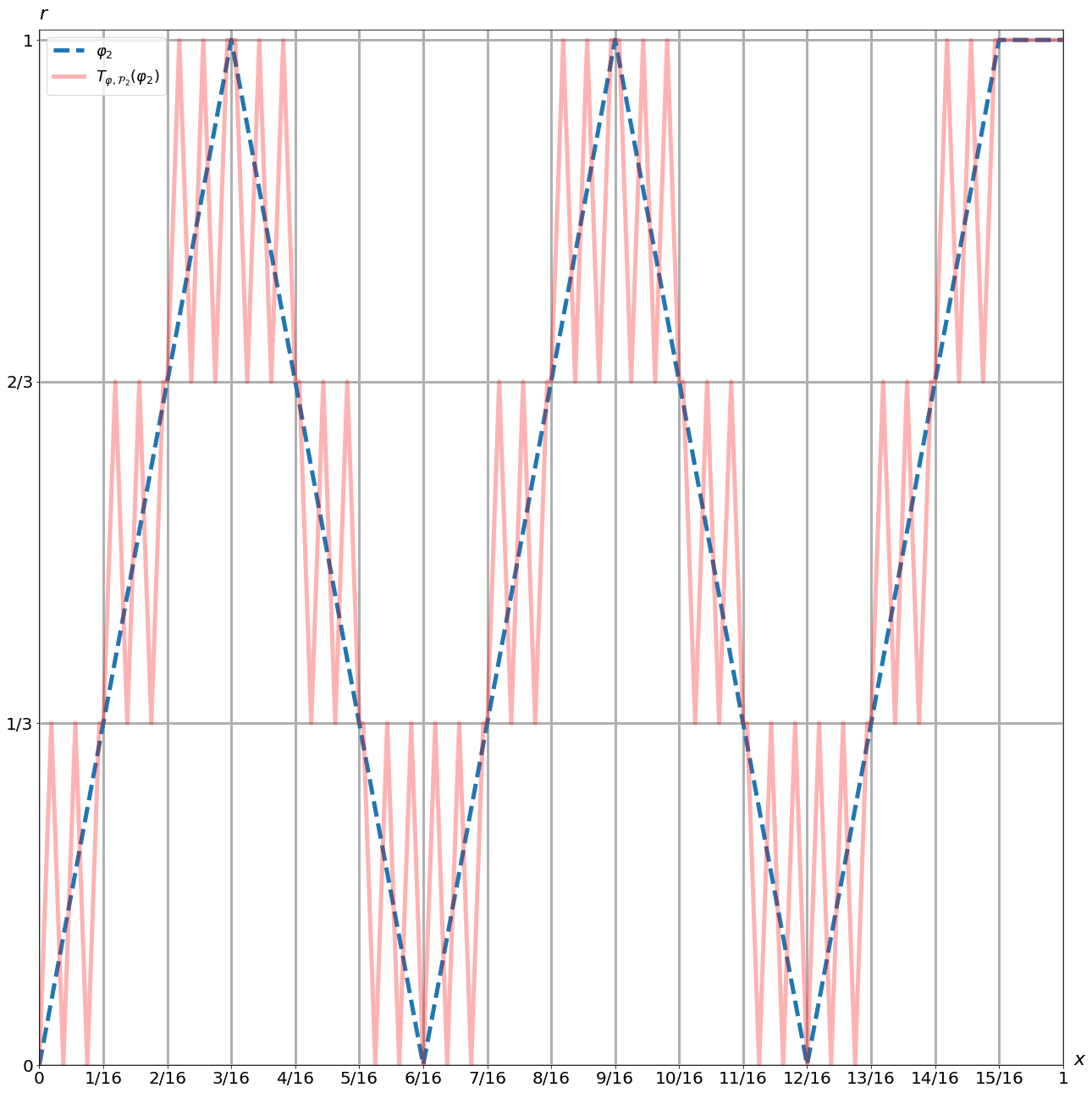} 
    \caption{In the foreground, the function $\varphi(x)=\varphi_n(x)$ for $n=2$ (blue, dashed line). In the background, $T_{\varphi,\capp_2}(\varphi)$ (red, simple line).  \label{fig:aux_function}}
\end{figure}

The function $\varphi$ and a partition $\capp$ of $[0,1]$ gives rise to an
 operator $T_{\varphi,\capp}$ which maps piecewise linear functions to piecewise linear functions.
Intuitively speaking, when we apply $T_{\varphi,\capp}$ to a piecewise linear function $f$, each growing linear piece of the graph of $f$ is replaced by the fitting affine transform of $\varphi$,
while each decreasing linear piece of the graph of $f$ is replaced by the fitting affine transform of $\varphi^{\mathrm{rev}}$ defined by $\varphi^{\mathrm{rev}}(x) = \varphi(1-x)$.
(By fitting, we mean that for any endpoint $x$ of any piece, $f(x)$ should equal $T_{\varphi,\capp}(f)(x)$.)
Formally, if $f:[0, 1] \to \mathbb{R}$ is a piecewise linear function, so that it is linear on any interval of the form
$[x_i, x_{i+1}]$ for the partition $\capp$ determined by $0=x_0 < x_1 < ... < x_m = 1$,
then for $x=c x_i + (1-c)x_{i+1}\in [x_i, x_{i+1}]$, let
$$T_{\varphi, \capp}(f)(x)= (f(x_{i+1})-f(x_i))\varphi(c) + f(x_i)$$
if $f$ is growing on $[x_i, x_{i+1}]$. Otherwise, let
$$T_{\varphi, \capp}(f)(x)= (f(x_i)-f(x_{i+1}))\varphi^{\mathrm{rev}}(c) + f(x_{i+1}).$$
We note that one could use the first formula for all the pieces to define $T_{\varphi,\capp}$ to get a very similar, but simpler operator. However, that would yield difficulties later, as we will emphasize.

Notice that $T_{\varphi,\capp}$ also depends on the partition, as we did not require the intervals  of monotonicity to  be maximal.
Now $\varphi = T_{\varphi,\capp_{1}}({ \mathrm {Id}}_{[0, 1]})$
when $\capp_{1}=\{ [0,1] \}.$
Using the partition $\capp_2=\mathcal{I}_{2n}$, the function $T_{\varphi,\capp_{2}}(\varphi)=T_{\varphi,\capp_{2}}\circ T_{\varphi,\capp_{1}}({ \mathrm {Id}}_{[0, 1]})$ 
is piecewise linear
with pieces determined by the partition $\mathcal{I}_{4n}$. By iterating this argument, one can recursively define 
${\mathbf T}_{\varphi,k}({ \mathrm {Id}}_{[0, 1]}):=T_{\varphi,\capp_k}\circ ...\circ T_{\varphi,\capp_j}\circ ... \circ T_{\varphi,\capp_1}({ \mathrm {Id}}_{[0, 1]})$ for $k\geq 2$,
using partitions $\capp_j=\mathcal{I}_{2(j-1)n}$.
By visualizing the particular functions of this sequence, one can easily verify that this sequence converges 
uniformly  as $k$ tends to infinity. (Again, we refer to Figure \ref{fig:aux_function}.)  Let us denote the limit by $\varphi^*=\varphi^*_n$. Finally, for $(x, y)\in \Delta$, let
$$\Phi_{n, k}(x, y)=\left({\mathbf T}_{\varphi,k}({ \mathrm {Id}}_{[0, 1]})\right)(y),$$
$$\Phi(x, y) = \Phi_n(x, y) = \varphi^*_n(y).$$


\begin{lemma} \label{lemma:high_box_holder}
    For $c = 6$ and $\alpha_n = \frac{\log 3}{2n \log 2}$, 
    we have $\varphi^* \in C_{6}^{\alpha_n}([0, 1])$. Consequently, $\Phi \in C_{6}^{\alpha_n}(\Delta)$.
\end{lemma}

\begin{proof}
    Let $x, y\in [0,1]$ be arbitrary, and fix $k$ so that
    $$\frac{1}{2^{2n(k+1)}} \leq |x-y|\leq \frac{1}{2^{2nk}}.$$
    Then $x, y$ are either in the same interval or in adjacent intervals determined by $\mathcal{I}_{2nk}$. As the $\varphi^*$-image of 
    such an interval has diameter $0$ or $\frac{1}{3^k}$, this yields
    $$|\varphi^*(x) - \varphi^*(y)|\leq \frac{2}{3^k} =  6\cdot \left(\frac{1}{2^{2n(k+1)}}\right)^{\frac{\log 3}{2n \log 2}} \leq
    c|x-y|^{\alpha_n}.$$
    The claim concerning $\Phi$ trivially follows.
\end{proof}

Note that  $\alpha_n\to 0+$ as $n\to \infty$, hence we will be able to use our construction to produce 
Hölder-$\alpha$ functions for any $\alpha$.


Now we would like to ensure that $\Phi_n$ has large level sets. To this end, we introduce some further notation. First, we will divide $\tau_{2n}$ into subsets based
on its "rows". More explicitly, we say that $T\in\tau_{2n}$ is in $\tau_{2n}(j)$ ($j=0, 1, ..., 2^{2n}-1$), if its projection to the $y$-axis is the segment $\left[1-\frac{j+1}{2^{2n}}, 1 - \frac{j}{2^{2n}}\right]$.
This reverse indexing makes the formulation of the following lemma simpler. 
To help to follow and check the details of the next argument for the case $n=2$ the reader can use 
Figure \ref{fig:sierpj}.

\begin{lemma}
	For $0\leq j <2^{2n}$, the cardinality $p(j)$ of $\tau_{2n}(j)$ is $2^m$, where $m$ is the number of 1s in the binary expansion of $j$. 
    (Thus $p(j)$ is indeed independent of $n$ under the condition $0\leq j <2^{2n}$, omitting $n$ from this notation is justified.)
\end{lemma}

\begin{proof}
	We use induction on $n$. For $n=0$, the claim trivially holds. Moreover, for $n>0$, if $j<2^{2n-1}$, then $p_{2n}(j)=p_{2(n-1)}(j)$, for which the claim holds by induction. On the other hand,
	if $j\geq 2^{2n-1}$, then $\tau_{2n}(j)$ consists of two identical copies of $\tau_{2n}(j-2^{2n-1})$, hence $p_{2n}(j) = 2p_{2n}(j-2^{2n-1})$. However, the binary expansion of $j$ contains one more 1s
	than the binary expansion of $j-2^{2n-1}$, thus it yields the desired conclusion.
\end{proof}

We also wish to count the cardinality of certain special unions of such $\tau_{2n}(j)$s. Notably, we are interested in the number of triangles in $\tau_{2n}$ which are mapped to the intervals $\left[\frac{i}{3}, \frac{i+1}{3}\right]$ 
($i=0, 1, 2$) by $\Phi_n$, or by $\Phi_{n, 1}$. (By construction, these sets of triangles coincide, as the images of triangles in $\tau_{2n}$ do not change as $k\to\infty$ in $\Phi_{n,k}$). 
Denote the number of such triangles by $x_{n,i}$ ($i=0, 1, 2$). 

\begin{lemma} \label{lemma:recursive}
	$x_{n,0} = 3^{2n-1}+3^{n-1} $, $x_{n,1} = 3^{2n-1}-3^{n-1}$, $x_{n,2} = 3^{2n-1}-1$.
\end{lemma}

\begin{proof}
	As $\Phi_{n, 1}(x, y)=\varphi_n(y)$, due to the construction of $\varphi_n$, we obtain the following identities (all the congruences are to be understood mod 6 in this proof,
	again we refer to Figure \ref{fig:sierpj}):
	\begin{equation}\label{eq:x0}
	x_{n,0} = \sum_{\substack{0<j<2^{2n} \\ j\equiv 3}} p(j) + \sum_{\substack{0<j<2^{2n} \\ j\equiv 4}} p(j) =: y_{n, 3} + y_{n, 4},
	\end{equation}
	\begin{equation}\label{eq:x1}
	x_{n,1} = \sum_{\substack{0<j<2^{2n} \\ j\equiv 2}} p(j) + \sum_{\substack{0<j<2^{2n} \\ j\equiv 5}} p(j) =: y_{n, 2} + y_{n, 5},
	\end{equation}
	\begin{equation}\label{eq:x2}
	x_{n,2} = \sum_{\substack{0<j<2^{2n} \\ j\equiv 0}} p(j) + \sum_{\substack{0<j<2^{2n} \\ j\equiv 1}} p(j) =: y_{n, 0} + y_{n, 1}.
	\end{equation}
	We will prove our claim via induction on $n$, to which end we first give recursive formulas for the sequences $(y_{n,i})$. By definition,
    $$y_{n+1, i} = \sum_{\substack{0<j<2^{2(n+1)} \\ j\equiv i}} p(j).$$
    Dividing the interval $(0, 2^{2(n+1)})$ into four equal parts, this can be rewritten as
    \begin{equation}\label{eq:yi}
    y_{n+1, i} = \sum_{\substack{0<j<2^{2n} \\ j\equiv i}} p(j) + \sum_{\substack{2^{2n}\leq j<2^{2n+1} \\ j\equiv i}} p(j) + \sum_{\substack{2^{2n+1}\leq j<2^{2n+1} + 2^{2n} \\ j\equiv i}} p(j) + \sum_{\substack{2^{2n+1} + 2^{2n}\leq j<2^{2(n+1)} \\ j\equiv i}} p(j).
    \end{equation}
    Now note that if $j\in \left[2^{2n}, 2^{2n+1}\right)$,
	then $p(j) = 2p(j-2^{2n})$, while if $j\in\left[2^{2n+1}, 2^{2n+1}+2^{2n}\right)$, then $p(j)=2p(j-2^{2n+1})$, and finally, if $j\in\left[2^{2n+1}+2^{2n}, 2^{2(n+1)}\right)$, then $p(j)=4p(j-2^{2n+1}-2^{2n})$. 
	We also know that how these subtractions interact with mod 6 residue classes, as $2^{2n}\equiv 4$, while $2^{2n+1}\equiv 2$.
	Thus \eqref{eq:yi} can be continued as
	$$y_{n+1, i} =\sum_{\substack{0<j<2^{2n} \\ j\equiv i}} p(j) + 2\sum_{\substack{0\leq j<2^{2n} \\ j\equiv i+2}} p(j) + 2\sum_{\substack{0\leq j<2^{2n} \\ j\equiv i+4}} p(j) + 4\sum_{\substack{0\leq j<2^{2n} \\ j\equiv i}} p(j).$$
\begin{figure}[ht] 
    \centering
    \includegraphics[scale=0.3]{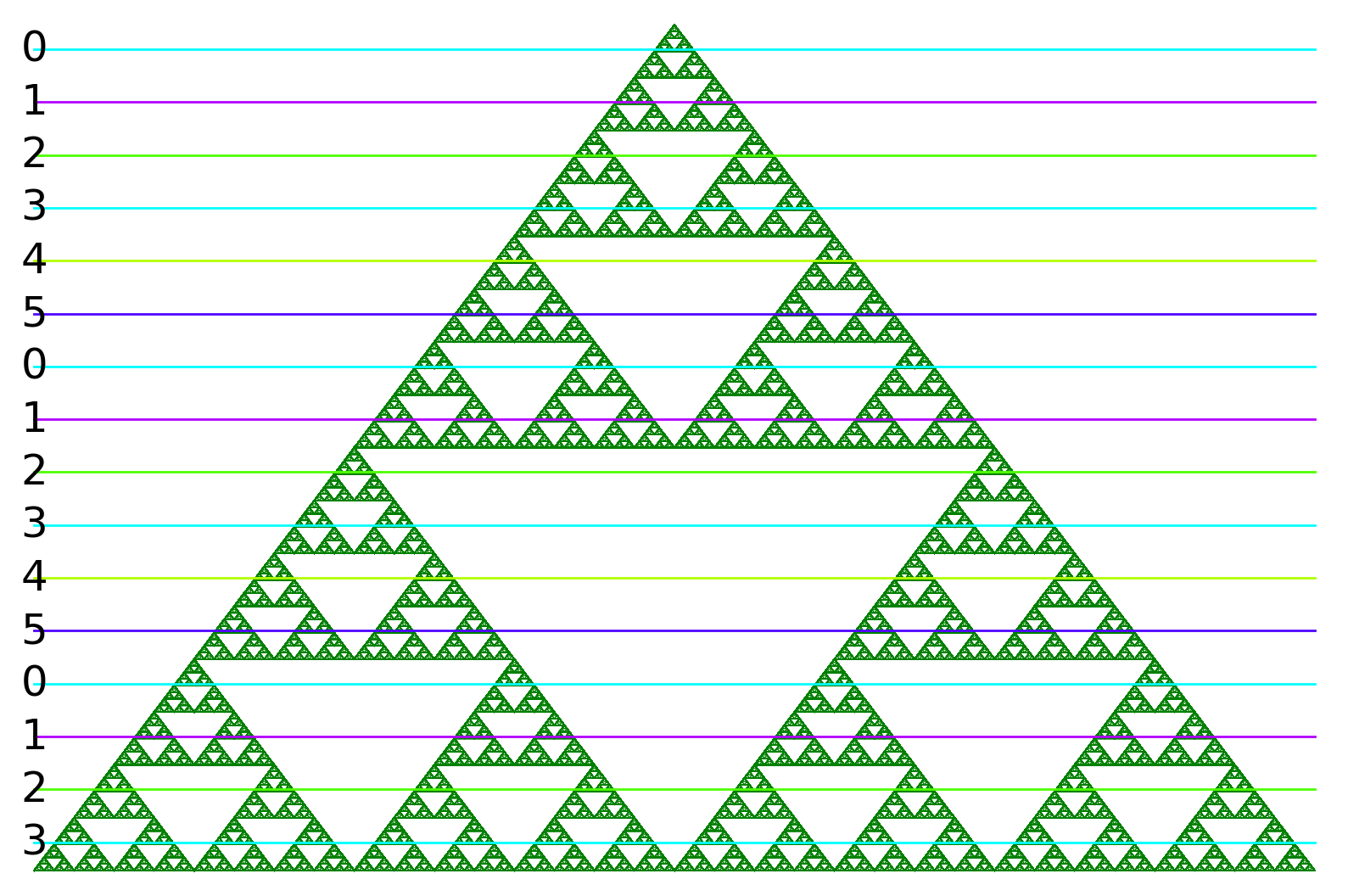} 
    \caption{Illustration of the case when $n=2$. The horizontal lines are numbered modulo $6$.\label{fig:sierpj}}
\end{figure}
	Notice that these sums resemble the sums defining $y_{n,i}, y_{n,i+2}, y_{n,i+4}, y_{n, i}$, respectively, apart from that they might include $p(0)=1$, depending on
    whether the congruence they are conditioned on holds for $j=0$. This observation leads to the following set of
	recursive equations:
	$$y_{n+1,0} = 5y_{n, 0} + 2y_{n,2} + 2y_{n,4} + 4,\qquad
	y_{n+1,1} = 5y_{n, 1} + 2y_{n,3} + 2y_{n,5},$$
	$$y_{n+1,2} = 5y_{n, 2} + 2y_{n,4} + 2y_{n,0} + 2,\qquad
	y_{n+1,3} = 5y_{n, 3} + 2y_{n,5} + 2y_{n,1},$$
	$$y_{n+1,4} = 5y_{n, 4} + 2y_{n,0} + 2y_{n,2} + 2,\qquad
	y_{n+1,5} = 5y_{n, 5} + 2y_{n,1} + 2y_{n,3}.$$
	Notice further that $y_{1,0} = y_{1,4} = y_{1,5} = 0$, $y_{1,1} = y_{1,2} = 2$, and $y_{1,3}=4$. 

	Now it is easy to check by induction, substituting the below formulae, that
	$$y_{n,0} = 3^{2(n-1)}-1,\qquad\qquad
	y_{n,1} = 2\cdot 3^{2(n-1)},\qquad
	y_{n,2} = 3^{2(n-1)}+3^{n-1},$$
	$$y_{n,3} = 2\cdot(3^{2(n-1)}+3^{n-1}),\quad
	y_{n,4} = 3^{2(n-1)}-3^{n-1},\quad
	y_{n,5} = 2\cdot(3^{2(n-1)}-3^{n-1}).$$
	Substituting into \eqref{eq:x0}, \eqref{eq:x1}, \eqref{eq:x2} yields the statement of the lemma.
\end{proof}

\begin{lemma} \label{lemma:sier_upper_box_dense_set}
    \begin{itemize}
        \item[(a)] If $r$ is a triadic rational, then $\overline{\dim}_B(\Phi_n^{-1}(r))=\frac{\log 3}{\log 2}$.
        \item[(b)] Otherwise, if $r=\sum_{i=1}^{\infty}r_i 3^{-i}$, then we have
        $$\overline{\dim}_B(\Phi_n^{-1}(r)) =\limsup_{k\to\infty} \frac{\log \prod_{i=1}^{k}x_{n, r_i}}{\log 2^{2nk}}.$$
        \item[(c)] $D_{\overline{B}}^{\Phi_n}(\Delta)\geq d_n$, where $d_n = \frac{\log \sqrt[3]{x_{n,0} x_{n,1} x_{n, 2}}}{\log 2^{2n}}$.
        \item[(d)] $d_n\geq \frac{\log 3^{2n-1}-1}{\log 2^{2n}}$.
    \end{itemize}

\end{lemma}

\begin{proof}
	Statement (a) follows from the fact that such values are taken on complete subtriangles of $\Delta$ by construction due to the horizontal parts of $\phi(x)$.
	
	Now let $r=\sum_{i=1}^{\infty}r_i 3^{-i}$ as in (b). Fix $k$, and for $r^{(k)}=\sum_{i=1}^{k}r_i 3^{-i}$, consider $I=\left[r^{(k)}, r^{(k)} + 3^{-k}\right]$. 
	Due to the way we defined the function $\Phi_{n,k}$, by induction on $k$, one can see that in this case, the number of triangles in $\tau_{2nk}$ which are 
    mapped to $I$ by $\Phi_{n,k}$ (and hence by $\Phi_n$)
	is precisely $\prod_{i=1}^{k}x_{n, r_i}$. (We remark that we use at this point that the operator $T_\varphi$ was defined using $\varphi^{\mathrm{rev}}$ instead of $\varphi$
    on decreasing intervals of its argument. Otherwise, for example inside $\left[0, \frac{1}{2^{2n}}\right]$, on which $\varphi$ is increasing, there would be $x_{n, r_i}$ base intervals of $\mathcal{I}_{4n}$ 
    mapped to $I$ for $r^{(2)} = 0 + 3^{-2}r_i$, while inside $\left[\frac{5}{2^{2n}}, \frac{6}{2^{2n}}\right]$, there would be $x_{n, 3-r_i}$ such intervals. This dichotomy would make calculations 
    much more complicated, which issue is prevented by the slightly more complicated definition of $T_\varphi$.)
    Moreover, as $r$ is in the interior of $I$, these are all the triangles which are intersected by $\Phi_n^{-1}(r)$.
	Hence if we use the sequence $(\tau_{2nk})_{k=1}^{\infty}$ as BDDS, we obtain the desired equation.
	
	As almost every number is normal, we know that for almost every $r$ the density of each ternary digit is $\frac{1}{3}$ in its ternary expansion. Thus for almost every $r$ we know that
	$$\sqrt[k]{\prod_{i=1}^{k}x_{n, r_i}}\to \sqrt[3]{x_{n,0} x_{n,1} x_{n, 2}},$$
	which yields (c).

    Statement (d) follows from $x_{n,0}x_{n,1}\geq x_{n,2}^2$, implying $\sqrt[3]{x_{n,0} x_{n,1} x_{n, 2}}\geq 3^{2n-1}-1$.
\end{proof}

Having this lemma, we can prove the following theorem:


\begin{theorem} \label{thm:sier_upper_box}
    Assume $\alpha < \alpha_n = \frac{\log 3}{2n \log 2}$. Then 
    $$ D_{\overline{B}*}(\alpha, \Delta) \geq \frac{\log \sqrt[3]{x_{n,0} x_{n,1} x_{n, 2}}}{\log 2^{2n}}=d_n\geq \frac{\log 3^{2n-1}-1}{\log 2^{2n}},$$
where 
$x_{n,0},\  x_{n,1},\  x_{n, 2}$ were defined in Lemma \ref{lemma:recursive}.

Moreover, for any $0<\alpha<1$, we have
    $\frac{\log 3}{\log 2} - 1 \leq D_{\overline{B}*}(\alpha, \Delta)$
\end{theorem}

\begin{proof}[Proof of Theorem \ref{thm:sier_upper_box}]
    The second statement is simply the application of Theorem \ref{thm:gen_lower_bound_box}. Thus it suffices to prove the first one concerning small $\alpha$s.

    Due to Lemma \ref{lemma:upper_box_generic}, it suffices to provide a dense
    set of functions such that for these functions almost every level set is large. 
    
    Due to Lemma \ref{lemma:high_box_holder}, we can fix $n$ such that
    $\Phi=\Phi_n$ is 6-Hölder-$\alpha$.

    Due to Lemma \ref{lemma:standard_piecewise_affine}, it suffices to prove that if $f_0$ is a standard strongly piecewise affine function
    and $\varepsilon>0$, we can find $f$ with $D_{\overline{B}*}^{f}(\Delta)\geq d_n 
	$ 
    such that $\|f-f_0\|<\varepsilon$.
    Let $M$ be the Lipschitz constant of $f_0$.
    Let $N \in \mathbb{N}$  be fixed  later such that $f_0$ is standard piecewise affine on the $N$th level.
    We will define $f$ so that
    it agrees with $f_0$ on $V(\tau_{2N})$. Moreover, for $T\in \tau_{2N}$, let $\Psi$ be a similarity of ratio $2^{2N}$ which maps $T$ to $\Delta_0$
    such that $f_0(v)=f_0(v')$ for some vertices $v, v'\in T$ if and only if $\Phi(\Psi(v))=\Phi(\Psi(v'))$. 
    Actually, one can find $a_N,b_N\in\mathbb{R}$
    with $f(v)=a_N\Phi(\Psi(v)) + b_N$ for any vertex $v\in T$, and $|a_N|<2^{-2N}M$ due to the Lipschitz property of $f_0$. Now extend $f$ to $T\cap \Delta$ by
    $f(x)=a_N\Phi(\Psi(x)) + b_N$. Due to the uniform continuity of $f_0$, this construction satisfies $\|f-f_0\|<\varepsilon$ 
    for large enough $N$.
    Moreover, for $T\in \tau_{2N}$ and arbitrary $x, y \in T$ we have
    $$|f(x) - f(y)|\leq 6\cdot a_N\cdot | \Psi(x) - \Psi(y)|^{\alpha} = 6\cdot a_N \cdot 2^{2N\alpha}|x-y|^{\alpha} = \mathrm o(N) |x-y|^{\alpha},$$
    due to $|a_N|<2^{-2N}M$ and $\alpha<1$. For large enough $N$, this easily yields that $f$ is 1-Hölder-$\alpha$ on $\Delta$, using that $f$ coincides with the
    $1^{-}$-Hölder-$\alpha$ function $f_0$ on $V(\tau_{2N})$.
    Finally, almost every level set consists of finitely many similar images
    of level sets of $\Phi$. That is, $D_{\overline{B}*}^{f}(\Delta)\geq d_n$ for a dense set of functions.
    It yields that $D_{\overline{B}*}(\alpha, \Delta) \geq d_{n}$.
\end{proof}

We conclude this section with three remarks concerning the above construction.

\begin{remark}
    We note that the above machinery can be applied with other choices of $\varphi_n$, which is piecewise linear so that base intervals of $\mathcal{I}_{2n}$ are mapped to the intervals
    $\left[\frac{i}{3}, \frac{i+1}{3}\right]$, and if $x_{n,0}, x_{n,1}, x_{n,2}$ denotes the number of triangles in $\tau_{2nk}$ which are mapped to each of these intervals, we can draw the conclusion
    $\frac{\log \sqrt[3]{x_{n,0} x_{n,1} x_{n, 2}}}{\log 2^{2n}} \leq D_{\overline{B}*}(\alpha, \Delta)$. Our goal is to set $x_{n,0}x_{n,1}x_{n,2}$ as large as possible.
    However, as $\varphi_n$ has to go from 0 to 1, which are at odd steps apart if we take
    steps of length $1/3$, and there are an even number of base intervals of $\mathcal{I}_{2n}$, $\varphi_n$ must be constant on at least one base intervals, and hence $x_{n,0} + x_{n,1}+ x_{n,2}\leq 3^{2n}-1$.
    To maximize the product, one should attempt to construct $\varphi_n$ so that in fact, $x_{n,0} + x_{n,1}+ x_{n,2} = 3^{2n}-1$, and the terms are as close to each other as possible. We do not know if
    our construction is optimal, however, it should be noted that it is very close to the theoretical maximum. Indeed, if $\alpha\to \alpha_n +$, by the upper bound provided by our theorem
    we know that the geometric mean of $x_0, x_1, x_2$ cannot exceed $3^{2n-1}$, while our choice yields a quantity exceeding $3^{2n-1}-1$, by (d) of Lemma \ref{lemma:sier_upper_box_dense_set}.
    Thus we cannot hope for a much better lower bound without using other ideas.  
\end{remark}
    
\begin{remark} \label{remark:bounds_close}
    While it is implicitly stated in the previous remark, we would like to emphasize and make more explicit that if $\alpha\to \alpha_n +$, 
    then the upper estimate and the lower esimate provided by our theorems are remarkably close to each other. Notably,
    due to (d) of Lemma \ref{lemma:sier_upper_box_dense_set}, we know that the lower bound exceeds $\frac{\log 3^{2n-1}-1}{\log 2^{2n}}$, while
    the upper bound tends to $\frac{\log 3^{2n-1}}{\log 2^{2n}}$. The difference of these bounds is $\frac{\log\frac{3^{2n-1}}{3^{2n-1}-1}}{\log 2^{2n}}$, where the numerator
    is at most $\frac{1}{3^{2n-1}-1}$, due to $\log(1+t)\leq t$. That is, the difference is at most $\frac{1}{(3^{2n-1}-1)\cdot 2n\cdot \log 2}$, which is exponentially small in $n$.
\end{remark}

\begin{remark} \label{remark:generalize_construction}
    To construct $\varphi_n$, one could use other uniform partitions of the range $[0,1]$ instead of the $\frac{1}{3}$-based. For example, one could construct a similar almost cyclic $\varphi$ such that
    base intervals of $\mathcal{I}_n$ are mapped to $\left[0, \frac{1}{2}\right]$ and to $\left[\frac{1}{2}, 1\right]$, slightly modified in the last base intervals
    to enforce $\varphi(0)=0$, $\varphi(1)=1$, as above. This construction yields strong lower bounds if
    the numbers of triangles mapped to $\left[0, \frac{1}{2}\right]$ and to $\left[\frac{1}{2}, 1\right]$ are roughly the same. However, it would not hold with this construction, the number of triangles of the two type
    would differ by a scalar multiplier independent of $n$. Thus one has to introduce a more complicated construction, however, it can be carried out, we left the details to the reader. This
    argument leads to a lower bound which is exponentially good near $\alpha_n = \frac{1}{n}$. We conjecture that by using the uniform partition of the range $[0, 1]$ to $k$ pieces,
    and using an appropriate piecewise linear base function,
    this can be extended to any $\alpha_n = \frac{\log k}{n \log 2}$. However, defining this base function, and carrying out a similar calculation to the one presented in Lemma \ref{lemma:recursive}
    does not seem to be practically feasible. It would be very interesting to find a general approach to make this line of thought applicable.
\end{remark}

\section{Open problems} \label{sec:open_problems}

In the previous sections we expressed our interest in a number of questions we cannot yet answer. Below we summarize these problems in a concise way.

\begin{question} \label{quest:wilder_functions}
    Is there a robust method to construct functions for which one can effectively estimate the quantity $D_*^f(F)$, 
    but whose level sets do not consist of planar cross-sections of $F$? 
    Lacking such constructions severely limits our ability to estimate $D_*(\alpha, F)$, $D_{\underline{B}*}(\alpha, F)$, $D_{\overline{B}*}(\alpha, F)$,
    as discussed after Theorem \ref{theorem:lower_box_sier}.
\end{question}

\begin{question}
    How much one can relax the condition of $F$ having nice connection type in Lemma \ref{lemma:upper_box_generic} and in Theorem \ref{thm:upper_box_generic}? 
    We have noted that it cannot be dropped completely.
\end{question}

\begin{question}
    How much one can relax the homogeneity conditions imposed on $F$ in Theorem \ref{thm:gen_lower_bound_box}? Prove or disprove that we need self-similar sets with finitely many directions for the validity
    of this theorem.
\end{question}

\begin{question}
    Can one reduce the gap between the lower and upper bounds on $D_{\overline{B}*}(\alpha, F)$ provided by Theorem \ref{thm:sier_upper_box}? 
    Remark \ref{remark:generalize_construction} strongly suggests that the lower bound can be improved significantly, however, different $\alpha$s need varying constructions. Can one produce
    a unified approach, which yields a family of functions through which we can estimate $D_{\overline{B}*}(\alpha, F)$ for every $\alpha$?
\end{question}


\bibliographystyle{amsplain} 
\bibliography{sierbox} 

\end{document}